\documentclass{amsart}
\usepackage{amssymb}
\usepackage{amscd}
\usepackage{verbatim}
\usepackage{epsfig}
\usepackage[colorlinks=true,linkcolor=red,citecolor=blue,breaklinks=true]{hyperref}

\begin{document}
\newcommand\Mand{\ \text{and}\ }
\newcommand\Mor{\ \text{or}\ }
\newcommand\Mfor{\ \text{for}\ }
\newcommand\Real{\mathbb{R}}
\newcommand\RR{\mathbb{R}}
\newcommand\im{\operatorname{Im}}
\newcommand\re{\operatorname{Re}}
\newcommand\sign{\operatorname{sign}}
\newcommand\sphere{\mathbb{S}}
\newcommand\BB{\mathbb{B}}
\newcommand\HH{\mathbb{H}}
\newcommand\dS{\mathrm{dS}}
\newcommand\ZZ{\mathbb{Z}}
\newcommand\NN{\mathbb{N}}
\newcommand\codim{\operatorname{codim}}
\newcommand\Sym{\operatorname{Sym}}
\newcommand\End{\operatorname{End}}
\newcommand\Span{\operatorname{span}}
\newcommand\Ran{\operatorname{Ran}}
\newcommand\ep{\epsilon}
\newcommand\Cinf{\cC^\infty}
\newcommand\dCinf{\dot \cC^\infty}
\newcommand\CI{\cC^\infty}
\newcommand\dCI{\dot \cC^\infty}
\newcommand\Cx{\mathbb{C}}
\newcommand\Nat{\mathbb{N}}
\newcommand\dist{\cC^{-\infty}}
\newcommand\ddist{\dot \cC^{-\infty}}
\newcommand\pa{\partial}
\newcommand\Card{\mathrm{Card}}
\renewcommand\Box{{\square}}
\newcommand\Ell{\mathrm{Ell}}
\newcommand\Char{\mathrm{Char}}
\newcommand\WF{\mathrm{WF}}
\newcommand\WFh{\mathrm{WF}_\semi}
\newcommand\WFb{\mathrm{WF}_\bl}
\newcommand\WFsc{\mathrm{WF}_\scl}
\newcommand\WFscb{\mathrm{WF}_{\scl,\bl}}
\newcommand\Vf{\mathcal{V}}
\newcommand\Vb{\mathcal{V}_\bl}
\newcommand\Vsc{\mathcal{V}_\scl}
\newcommand\Vscsus{\mathcal{V}_\scsus}
\newcommand\Vz{\mathcal{V}_0}
\newcommand\Hb{H_{\bl}}
\newcommand\Hbh{H_{\bl,\semi}}
\newcommand\bHb{\bar H_{\bl}}
\newcommand\dHb{\dot H_{\bl}}
\newcommand\Hbb{\tilde H_{\bl}}
\newcommand\Hsc{H_{\scl}}
\newcommand\Hsch{H_{\scl,\semi}}
\newcommand\Hscb{H_{\scbl}}
\newcommand\Hscbh{H_{\scbl,\semi}}
\newcommand\bHscb{\bar H_{\scbl}}
\newcommand\dHscb{\dot H_{\scbl}}
\newcommand\Hscsus{H_{\scsus}}
\newcommand\Ker{\mathrm{Ker}}
\newcommand\Range{\mathrm{Ran}}
\newcommand\Hom{\mathrm{Hom}}
\newcommand\Id{\mathrm{Id}}
\newcommand\sgn{\operatorname{sgn}}
\newcommand\ff{\mathrm{ff}}
\newcommand\tf{\mathrm{tf}}
\newcommand\esssupp{\operatorname{esssupp}}
\newcommand\supp{\operatorname{supp}}
\newcommand\vol{\mathrm{vol}}
\newcommand\Diff{\mathrm{Diff}}
\newcommand\Diffd{\mathrm{Diff}_{\dagger}}
\newcommand\Diffs{\mathrm{Diff}_{\sharp}}
\newcommand\Diffb{\mathrm{Diff}_\bl}
\newcommand\Diffbh{\mathrm{Diff}_{\bl,\semi}}
\newcommand\Diffsc{\mathrm{Diff}_\scl}
\newcommand\Diffsch{\mathrm{Diff}_{\scl,\semi}}
\newcommand\Diffscsus{\mathrm{Diff}_\scsus}
\newcommand\DiffbI{\mathrm{Diff}_{\bl,I}}
\newcommand\Diffbeven{\mathrm{Diff}_{\bl,\even}}
\newcommand\Diffz{\mathrm{Diff}_0}
\newcommand\Psih{\Psi_{\semi}}
\newcommand\Psihcl{\Psi_{\semi,\cl}}
\newcommand\Psisc{\Psi_\scl}
\newcommand\Psisch{\Psi_{\scl,\semi}}
\newcommand\Psiscc{\Psi_\sccl}
\newcommand\Psiscch{\Psi_{\sccl,\semi}}
\newcommand\Psiscb{\Psi_\scbl}
\newcommand\Psiscbh{\Psi_{\scbl,\semi}}
\newcommand\Psib{\Psi_\bl}
\newcommand\Psibh{\Psi_{\bl,\semi}}
\newcommand\Psibc{\Psi_{\mathrm{bc}}}
\newcommand\Psibch{\Psi_{\mathrm{bc},\semi}}
\newcommand\Psibcdelta{\Psi_{\mathrm{bc},\delta}}
\newcommand\TbC{{}^{\bl,\Cx} T}
\newcommand\Tb{{}^{\bl} T}
\newcommand\Sb{{}^{\bl} S}
\newcommand\Tsc{{}^{\scl} T}
\newcommand\Tscsus{{}^{\scsus} T}
\newcommand\Ssc{{}^{\scl} S}
\newcommand\Sscsus{{}^{\scsus} S}
\newcommand\Lambdab{{}^{\bl} \Lambda}
\newcommand\zT{{}^{0} T}
\newcommand\Tz{{}^{0} T}
\newcommand\zS{{}^{0} S}
\newcommand\dom{\mathcal{D}}
\newcommand\cA{\mathcal{A}}
\newcommand\cB{\mathcal{B}}
\newcommand\cE{\mathcal{E}}
\newcommand\cG{\mathcal{G}}
\newcommand\cH{\mathcal{H}}
\newcommand\cU{\mathcal{U}}
\newcommand\cO{\mathcal{O}}
\newcommand\cF{\mathcal{F}}
\newcommand\cM{\mathcal{M}}
\newcommand\cQ{\mathcal{Q}}
\newcommand\cR{\mathcal{R}}
\newcommand\cI{\mathcal{I}}
\newcommand\cL{\mathcal{L}}
\newcommand\cK{\mathcal{K}}
\newcommand\cC{\mathcal{C}}
\newcommand\cX{\mathcal{X}}
\newcommand\cY{\mathcal{Y}}
\newcommand\cXsus{\mathcal{X}_\sus}
\newcommand\cYsus{\mathcal{Y}_\sus}
\newcommand\cP{\mathcal{P}}
\newcommand\cS{\mathcal{S}}
\newcommand\cZ{\mathcal{Z}}
\newcommand\cW{\mathcal{W}}
\newcommand\Ptil{\tilde P}
\newcommand\ptil{\tilde p}
\newcommand\chit{\tilde \chi}
\newcommand\yt{\tilde y}
\newcommand\zetat{\tilde \zeta}
\newcommand\xit{\tilde \xi}
\newcommand\taut{{\tilde \tau}}
\newcommand\phit{{\tilde \phi}}
\newcommand\mut{{\tilde \mu}}
\newcommand\taubsemi{\tau_{\bl,\hbar}}
\newcommand\lambdasemi{\lambda_{\hbar}}
\newcommand\sigmat{{\tilde \sigma}}
\newcommand\sigmah{\hat\sigma}
\newcommand\zetah{\hat\zeta}
\newcommand\etah{\hat\eta}
\newcommand\taub{\tau_{\bl}}
\newcommand\mub{\mu_{\bl}}

\newcommand\tauh{\tau_\semi}
\newcommand\muh{\mu_\semi}
\newcommand\taubh{\tau_{\bl,\semi}}
\newcommand\mubh{\mu_{\bl,\semi}}

\newcommand\nuh{\hat\nu}
\newcommand\loc{\mathrm{loc}}
\newcommand\compl{\mathrm{comp}}
\newcommand\reg{\mathrm{reg}}
\newcommand\GBB{\textsf{GBB}}
\newcommand\GBBsp{\textsf{GBB}\ }
\newcommand\bl{{\mathrm b}}
\newcommand\scl{{\mathrm{sc}}}
\newcommand\scbl{{\mathrm{sc,b}}}
\newcommand\sccl{{\mathrm{scc}}}
\newcommand\scsus{{\mathrm{sc-sus}}}
\newcommand\sus{{\mathrm{sus}}}
\newcommand{\sH}{\mathsf{H}}
\newcommand{\cte}{\digamma}
\newcommand\cl{\mathrm{cl}}
\newcommand\hsf{\mathcal{S}}
\newcommand\Div{\operatorname{div}}
\newcommand\hilbert{\mathfrak{X}}
\newcommand\smooth{\mathcal{J}}
\newcommand\decay{\ell}
\newcommand\symb{j}

\newcommand\Hh{H_{\semi}}

\newcommand\bM{\bar M}
\newcommand\Xext{X_{-\delta_0}}

\newcommand\xib{{\underline{\xi}}}
\newcommand\etab{{\underline{\eta}}}
\newcommand\zetab{{\underline{\zeta}}}

\newcommand\xibh{{\underline{\hat \xi}}}
\newcommand\etabh{{\underline{\hat \eta}}}
\newcommand\zetabh{{\underline{\hat \zeta}}}

\newcommand\zn{z}
\newcommand\sigman{\sigma}
\newcommand\psit{\tilde\psi}
\newcommand\rhot{{\tilde\rho}}

\newcommand\hM{\hat M}

\newcommand\Op{\operatorname{Op}}
\newcommand\Oph{\operatorname{Op_{\semi}}}

\newcommand\innr{{\mathrm{inner}}}
\newcommand\outr{{\mathrm{outer}}}
\newcommand\full{{\mathrm{full}}}
\newcommand\semi{\hbar}

\newcommand\Feynman{\mathrm{Feynman}}
\newcommand\future{\mathrm{future}}
\newcommand\past{\mathrm{past}}

\newcommand\elliptic{\mathrm{ell}}

\newcommand\even{\mathrm{even}}

\newcommand\sob{s}

\newtheorem{lemma}{Lemma}[section]
\newtheorem{prop}[lemma]{Proposition}
\newtheorem{thm}[lemma]{Theorem}
\newtheorem{cor}[lemma]{Corollary}
\newtheorem{result}[lemma]{Result}
\newtheorem*{thm*}{Theorem}
\newtheorem*{prop*}{Proposition}
\newtheorem*{cor*}{Corollary}
\newtheorem*{conj*}{Conjecture}
\numberwithin{equation}{section}
\theoremstyle{remark}
\newtheorem{rem}[lemma]{Remark}
\newtheorem*{rem*}{Remark}
\theoremstyle{definition}
\newtheorem{Def}[lemma]{Definition}
\newtheorem*{Def*}{Definition}

\newcommand{\mar}[1]{{\marginpar{\sffamily{\scriptsize #1}}}}
\newcommand\av[1]{\mar{AV:#1}}

\renewcommand{\theenumi}{\roman{enumi}}
\renewcommand{\labelenumi}{(\theenumi)}

\title[Limiting absorption principle, a Lagrangian approach]{Limiting
  absorption principle on Riemannian scattering (asymptotically conic)
  spaces, a Lagrangian approach}
\author[Andras Vasy]{Andr\'as Vasy}
\address{Department of Mathematics, Stanford University, CA 94305-2125, USA}

\email{andras@math.stanford.edu}


\thanks{The author gratefully
  acknowledges partial support from the NSF under grant numbers
  DMS-1361432 and DMS-1664683 and from a Simons Fellowship.}
\date{\today. Original version: May 29, 2019}

\begin{abstract}
We use a Lagrangian perspective to show the limiting absorption principle on Riemannian scattering, i.e.\
asymptotically conic, spaces, and their generalizations. More
precisely we show that, for non-zero spectral parameter, the `on
spectrum', as well as the `off-spectrum', spectral family is Fredholm
in function spaces which encode the Lagrangian regularity of generalizations of
`outgoing spherical waves' of scattering theory, and indeed this
persists in the `physical half plane'.
\end{abstract}

\maketitle

\section{Introduction and outline}
The purpose of this paper is to prove the limiting absorption
principle, concerning the limit of the resolvent at the spectrum on
appropriate function spaces, for Laplace-like operators on Riemannian scattering
(asymptotically conic at infinity) spaces $(X,g)$ using a description that
focuses on the outgoing radial set, which in phase space corresponds
to the well-known outgoing spherical waves in Euclidean scattering
theory. Thus, the result is a precise description of the
limiting resolvent in terms of mapping properties on spaces of (finite regularity)
Lagrangian distributions, where now the Lagrangian is conic in the
base manifold, rather than the fibers of the cotangent bundle as
familiar from standard microlocal analysis. Such a
result is well suited for the analysis of waves, especially at the
`radiation face', or `scri', see \cite{Haefner-Hintz-Vasy:Linear}, though we do not pursue this aspect
here. We explain more of the historic context of Lagrangian analysis
in scattering theory below, but already remark that recently such a
Lagrangian analysis proved very effective in the description of
internal waves in fluids by Dyatlov and Zworski \cite{Dyatlov-Zworski:Forced}.

The basic setting is Melrose's scattering pseudodifferential algebra
$\Psisc^{*,*}(X)$, see \cite{RBMSpec}, which for $X$ the radial
compactification of $\RR^n$ (to a ball) goes back to Parenti and Shubin
\cite{Parenti:Operatori,Shubin:Pseudo}, and which corresponds to any standard
quantization of symbols with the property
$$
|D_z^\alpha D_\zeta^\beta a(z,\zeta)|\leq
C_{\alpha\beta}\langle z\rangle^{r-|\alpha|}\langle \zeta\rangle^{s-|\beta|},
$$
with $r$ the decay order and $s$ the differential order. The key
property of this algebra is that the principal symbol is taken modulo
$\langle z\rangle^{-1}\langle\zeta\rangle^{-1}$ better terms, thus
also captures decay at infinity; see Section~\ref{sec:pseudo} for
more detail.

With this in mind, recall first that for $\sigma\neq
0$ real, elements of the spectral family $\Delta_g-\sigma^2$ are not
elliptic in this algebra due to the part of the principal symbol
capturing decay (essentially as $|\zeta|^2-\sigma^2$ can vanish), rather have a non-degenerate real principal symbol with a
source-to-sink Hamilton flow within their characteristic set (the zero
set of the principal symbol). One obtains a Fredholm problem
using variable decay order weighted scattering Sobolev spaces (which
are the standard Sobolev space on $\RR^n$, albeit of a microlocally
variable order), where
the order only matters on the characteristic set, needs to be monotone
along the Hamilton flow, and be greater than a threshold value ($-1/2$ in the
standard case, for the domain of the operator; the target space has
one additional order of decay) for one of the
radial sets (source or sink), which we call the incoming one, and less
than a threshold value for the other one, which we call the outgoing
one, see \cite{Vasy:Minicourse}, and see also \cite{Faure-Sjostrand:Upper}
for a semiclassical version in a dynamical systems setting. Moreover, this gives the limiting
absorption resolvent where the $+i0$ vs.\ $-i0$ limits (in terms of the
spectral parameter $\sigma^2$, thus $\im\sigma^2=2i\re\sigma\im\sigma$
shows that $\im\sigma\geq 0$ corresponds to the $+i0$ limit if
$\re\sigma>0$, and the $-i0$ limit if $\re\sigma<0$) correspond to
propagating estimates forward along the Hamilton flow, i.e.\ having high decay order at the
source, vs.\ propagating estimates backwards, i.e.\ having high decay
order at the sink. This can then be extended uniformly
to zero energy, see \cite{Vasy:Zero-energy}, using second microlocal
methods discussed below.

A different way of arranging a Fredholm setup is by considering a fixed decay
order Sobolev space which is lower than the threshold order, but
adding to it extra Lagrangian regularity relative to elements of $\Psisc^{1,1}$
characteristic on the outgoing radial set (referred to as `module
regularity', see \cite{Hassell-Melrose-Vasy:Spectral,
  Hassell-Melrose-Vasy:Microlocal, Haber-Vasy:Radial}, see also \cite{Dyatlov-Zworski:Forced}). (Since the
Lagrangian is at finite points in the fibers of the scattering
cotangent bundle, i.e.\ where $\zeta$ is finite in the Euclidean picture, the differential order is immaterial; only the decay
order matters.) For instance, the
background decay order can be taken $1/2$ less than the threshold, and
one may require regularity under one such pseudodifferential
factor. This makes the space to have $1$ order more decay everywhere
except at the outgoing radial set; since the thresholds are the same
in this case at both radial sets, this means that we have $1/2$ higher
order decay at the incoming radial set than the threshold. Very
concretely, this can be arranged using operators
$$
x^{-1}(x^2D_x+\sigma),\ x^{-1}(xD_{y_j})=D_{y_j},
$$
with $x$ the boundary defining
function, $y_j$ local coordinates on the boundary, and the metric is
to leading order warped product type relative to these. (So in the asymptotically
Euclidean setting, one could have $x=r^{-1}$, and $y$ local
coordinates on the sphere with respect to the standard spherical
coordinate decomposition.) Thus, the domain
space is the modified version of
$$
\{u\in x^{-1}L^2:\ x^{-1}(x^2D_x+\sigma)u,D_{y_j}u\in x^{-1}L^2\},
$$
with the modification just so that the operator maps it to the target
which simply has $1$ additional order of decay
\begin{equation}\label{eq:target-sp-intro}
\{u\in x^{0}L^2:\ x^{-1}(x^2D_x+\sigma)u,D_{y_j}u\in x^{0}L^2\}.
\end{equation}
Using the variable order Fredholm theory it is straightforward
to show (using a variable order that is in $(-1/2,0)$ at the incoming
radial set, and is $<-1$ at the outgoing radial set) that the outgoing
inverse indeed has the property that under this additional regularity
of the input (in the target space), the output lies in the additional
regularity domain. However, it is harder to directly run Fredholm
arguments since these involve duality and inversion, and the
additional module regularity gives a dual space for which it is harder to prove
estimates since the dual of, for instance, the space \eqref{eq:target-sp-intro} is
$$
x^{0}L^2 +x^{-1}(x^2D_x+\sigma) x^{0}L^2+\sum_j D_{y_j}x^{0}L^2,
$$
see \cite[Appendix~A]{Melrose-Vasy-Wunsch:Corners}.

A way around this difficulty with dualization, which we pursue in this
paper, is to use even stronger,
second microlocal, spaces, see \cite[Section~5]{Vasy:Zero-energy} in this
scattering context, and see
\cite{Bony:Second,Sjostrand-Zworski:Fractal,Vasy-Wunsch:Semiclassical} in different
contexts. Recall that these second microlocal techniques play a role
in precise analysis at a Lagrangian, or more generally coisotropic, submanifold. These second microlocal
techniques were employed in \cite{Vasy:Zero-energy} due to the
degeneration of the principal symbol at zero energy, corresponding to
the quadratic vanishing of any dual metric function at the zero
section; the chosen Lagrangian is thus the zero section, really
understood as the zero section at infinity. In
a somewhat simpler way than in other cases, this second microlocalization at the zero
section is accomplished by simply using the
b-pseudodifferential operator algebra of Melrose \cite{Melrose:Atiyah}. In an informal way, this arises
by blowing up the zero section of the scattering cotangent bundle at
the boundary, though a more precise description (in that it makes
sense even at the level of
quantization, the spaces themselves are naturally diffeomorphic)  is
the reverse: blowing up the corner (fiber infinity over the boundary)
of the b-cotangent bundle: see Section~\ref{sec:pseudo} for more
detail and additional references. (But the basic point is that the scattering
vector fields $x^2D_x,xD_{y_j}$ are replaced by totally
characteristic, or b-, vector fields $xD_x,D_{y_j}$.) In \cite{Vasy:Zero-energy} this was used to
show a uniform version of the resolvent estimates down to zero energy
using variable differential order b-pseudodifferential
operators. Indeed, the
differential order of these, cf.\ the aforementioned blow-up of the
corner, corresponds to the scattering decay order away from the zero
section, thus this allows the uniform analysis of the problem to zero energy.
However, here the decay order (of the b-ps.d.o.) is also crucial, for it corresponds to
the spaces on which the exact zero energy operator (i.e.\ with $\sigma=0$) is Fredholm, which,
with $\Hb$ denoting weighted b-Sobolev spaces relative to the
scattering (metric) $L^2$-density, are $\Hb^{\tilde r,l}\to\Hb^{\tilde
  r-2,l+2}$ with $|l+1|<\frac{n-2}{2}$, where $\tilde r$ is the
variable order (which is irrelevant at zero energy since the operator
is elliptic in the b-pseudodifferential algebra then). (The more
refined, fully 2-microlocal, spaces $\Hscb^{s,r,l}$, see Section~\ref{sec:pseudo}, corresponding to the blow-up of
the corner, have three orders: sc-differential $s$,
sc-decay/b-differential $r$ and b-decay $l$; using all of these is convenient,
as the operators are sc-differential-elliptic, so one can use easily that
this order, $s$, is essentially irrelevant; this modification is not crucial.)

Now, for $\sigma\neq 0$ real, one can work in a second microlocal
space by simply conjugating the spectral family $P(\sigma)$ by
$e^{i\sigma/x}$ (this being the multiplier from the right),
with the point being that this conjugation acts as a canonical
transformation of the scattering cotangent bundle, moving the outgoing
radial set to the zero section, see Sections~\ref{sec:pseudo}-\ref{sec:operator}. Then the general second microlocal analysis
becomes b-analysis. Indeed, note that this conjugation
moves
$$
x^{-1}(x^2D_x+\sigma),\ \text{resp.}\ x^{-1}(xD_{y_j}),
$$
to
$$
x^{-1}(x^2D_x)=xD_x,\ \text{resp.}\ x^{-1}(xD_{y_j})=D_{y_j},
$$
so the
Lagrangian regularity becomes b-differential-regularity indeed. Notice
that the conjugate of the simplest model operator
$$
P(\sigma)=(x^2D_x)^2+i(n-1)x(x^2 D_x)+x^2\Delta_h-\sigma^2\in \Diffsc^2(X)\subset\Diffb^2(X),
$$
which is the Laplacian of the conic metric
$g_0=x^{-4}\,dx^2+x^{-2}h(y,dy)$ (considered near the `large end', $x=0$),
is then
\begin{equation*}\begin{aligned}
\hat
P(\sigma)&=e^{-i\sigma/x}P(\sigma)e^{i\sigma/x}=(x^2D_x-\sigma)^2+i(n-1)x(x^2
D_x-\sigma)+x^2\Delta_y-\sigma^2\\
&=(x^2D_x)^2-2\sigma(x^2D_x)+i(n-1)x(x^2
D_x)-i(n-1)x\sigma +x^2\Delta_y\\
&\qquad\qquad\in x\Diffb^2(X),
\end{aligned}\end{equation*}
which has one additional order of vanishing in this b-sense (the
factor of $x$ on the right). (This is
basically the effect of the zero section of the sc-cotangent bundle
being now in the characteristic set.) Moreover, to leading order in
terms of the b-decay sense, i.e.\ modulo $x^2\Diffb^2(X)$, this is the
simple first order operator
$$
-2\sigma x\Big(xD_x+i\frac{n-1}{2}\Big).
$$
(In general, decay is
controlled by the normal operator of a b-differential operator,
which arises by setting $x=0$ in its coefficients after factoring out an
overall weight, and where one thinks of it as acting on functions on
$[0,\infty)_x\times\pa X$, of which $[0,\delta_0)_x\times\pa X$ is
identified with a neighborhood of $\pa X$ in $X$.)
This is non-degenerate for $\sigma\neq 0$ in that, on suitable spaces,
it has an invertible normal operator; of course, this is not an
elliptic operator, so some care is required. Notice that terms like
$(x^2D_x)^2$ and $\sigma x^2D_x$ have the same scattering decay order,
i.e.\ on the front face of the blown up b-corner they are equally
important. Thus, we use real principal type plus radial points
estimates at finite points in the scattering cotangent bundle,
together with a radial point type analysis of the zero section, but
now interpreted in the second microlocal setting. This gives, for the
general class of operators discussed in Section~\ref{sec:operator},
which includes the spectral family of the Laplacian of Riemannian
scattering metrics, with $\im\alpha_\pm(\sigma)=0$ in the case of the operator
discussed above:

\begin{thm}\label{thm:main}
Suppose that $P(\sigma)$ satisfies the hypotheses of
Section~\ref{sec:operator} and let $\alpha_+(\sigma)$,
$\alpha_-(\sigma)$ be as given there, see
\eqref{eq:actual-normal-op-hat} and \eqref{eq:other-normal-op-hat}; thus,
$\im\alpha_\pm(\sigma)=0$ if $P(\sigma)$ is formally
self-adjoint, and $\mp 2\sigma x\alpha_\pm(\sigma)$ is the subprincipal
symbol at $\mp \sigma\,\frac{dx}{x^2}$. Suppose also that
\begin{equation*}
\tilde
r+\ell+1/2-\im\alpha_-(\sigma)>0,\qquad\ell+1/2-\im\alpha_+(\sigma)<0,
\end{equation*}
and $K$ a compact subset of $\{\sigma\in\Cx:\
\im\sigma\geq 0,\ \sigma\neq 0\}$. For
$\sigma\in K$, let
$$
\hat
P(\sigma)=e^{-i\sigma/x}P(\sigma)e^{i\sigma/x}.
$$

Then
$$
\hat P(\sigma):\{u\in\Hb^{\tilde r,\ell}:\ \hat P(\sigma)u\in \Hb^{\tilde r,\ell+1}\}\to\Hb^{\tilde r,\ell+1}
$$
is Fredholm, and if $P(\sigma)=P(\sigma)^*$ for $\sigma\in\RR\setminus\{0\}$ then it is invertible, with this inverse being the $\pm i0$ resolvent limit
(in the sense of $\sigma^2\pm i0$) of
$P(\sigma)$ corresponding to $\pm\re\sigma>0$, and the
norm of $\hat P(\sigma)^{-1}$ as an element of $\cL(\Hb^{\tilde
  r,\ell+1},\Hb^{\tilde r,\ell})$ is uniformly bounded for
$\sigma\in K$.
Furthermore, invertibility is preserved under suitably small
perturbations of $P(\sigma)$.

These statements also hold if both inequalities on the orders are
reversed:
\begin{equation*}
\tilde
r+\ell+1/2-\im\alpha_-(\sigma)<0,\qquad
\ell+1/2-\im\alpha_+(\sigma)>0,
\end{equation*}
provided one also reverses the sign of
$\im\sigma$ to $\im\sigma\leq 0$, and thus takes
$\{\sigma\in\Cx:\
\im\sigma\leq 0,\ \sigma\neq 0\}$ above.

Furthermore, the statements hold on second microlocal spaces, recalled
in Section~\ref{sec:pseudo},
$$
\hat P(\sigma):\{u\in\Hscb^{s,r,\ell}:\ \hat P(\sigma)u\in \Hscb^{s-2,r+1,\ell+1}\}\to\Hscb^{s-2,r+1,\ell+1}
$$
with
$$
\ell+1/2-\im\alpha_+(\sigma)<0,\qquad r+1/2-\im\alpha_-(\sigma)>0,
$$
as well as with
$$
\ell+1/2-\im\alpha_+(\sigma)>0,\qquad r+1/2-\im\alpha_-(\sigma)<0
$$
(again reversing the sign of $\im\sigma$).
\end{thm}

\begin{rem}
Note that $\Hb^{\tilde r,\ell}=\Hscb^{\tilde r,\tilde r+\ell,\ell}$, so $\tilde
r+\ell$ is the scattering decay order away from the zero section. Thus
the statements on $\Hb$ and $\Hscb$ spaces in the theorem are very
similar, including in terms of the restrictions on the orders, with
the main advantage of the $\Hscb$ statements being the ability to use
ellipticity in the sc-differential sense, making the order $s$ arbitrary.
\end{rem}

\begin{rem}
Here $\alpha_\pm(\sigma)$ are functions on $\pa X$, and the stated
inequalities, such as $\ell+1/2-\im\alpha_+(\sigma)<0$, are assumed to
hold {\em at every point} on $\pa X$.

In the case of the vector valued version, i.e.\ if $P(\sigma)$ acts on
sections of a vector bundle equipped with a fiber inner product, such as on scattering one-forms or
symmetric scattering 2-cotensors, the statement and the proof are
completely parallel, with the only change that now $\alpha_\pm(\sigma)$ are
valued in endomorphisms, and the inequalities involving $\alpha_\pm$
are understood in the sense of bounds for endomorphisms (such as
positive definiteness).
\end{rem}

\begin{rem}
We in fact show regularity statements below of the kind that if
$u\in\Hscb^{s',r',\ell}$ with $r'$ satisfying an inequality like $r$, and
if $\hat P(\sigma)u\in\Hscb^{s-2,r+1,\ell+1}$, then
$u\in\Hscb^{s,r,\ell}$, and the estimate for $u$ in terms of $\hat
P(\sigma)u$ (and a relatively compact term) implied by the Fredholm property
holds. See for instance Proposition~\ref{prop:real-sigma-Fredholm}.

One can also improve the b-decay order $\ell$; see Remark~\ref{rem:decay-improve}.
\end{rem}

Notice that, in terms of the limiting absorption principle, there are
two ways to implement this conjugation: one can conjugate either by
$e^{i\sigma/x}$, where $\sigma$ is now complex, or by
$e^{i\re\sigma/x}$. The former, which we pursue, gives much stronger
spaces when $\sigma$ is not real with $\im\sigma>0$ (which is from where we
take the limit), as $e^{i\sigma/x}$ entails an
exponentially decaying weight $e^{-\im\sigma/x}$, so if the original
operator is applied to $u$, the conjugated operator is applied to
$e^{\im\sigma/x} u$ times an oscillatory factor.

We also note that under non-trapping assumptions, mutatis mutandis,
all the arguments extend to the large $\sigma$ (with $\im\sigma$
bounded) setting via a semiclassical version of the argument presented
below, as we show in Section~\ref{sec:high}, namely one has

\begin{thm}\label{thm:high}
With $\hat P(\sigma)$ as above,
\begin{equation*}
\tilde
r+\ell+1/2-\im\alpha_-(\sigma)>0,\quad\ell+1/2-\im\alpha_+(\sigma)<0,\quad
r+1/2-\im\alpha_-(\sigma)>0,
\end{equation*}
and under the
additional assumption that the bicharacteristic flow is non-trapping,
for and $\sigma_0>0$ there is $C>0$ such that high energy
estimates hold on the semiclassical spaces, $h=|\sigma|^{-1}$,
$\im\sigma\geq 0$:
$$
\|u\|_{\Hbh^{\tilde r,l}}\leq C|\sigma|^{-1}\|\hat
P(\sigma)u\|_{\Hbh^{\tilde r,l+1}}
$$
and
$$
\|u\|_{\Hscbh^{s,r,l}}\leq C|\sigma|^{-1}\|\hat P(\sigma)u\|_{\Hscbh^{s-2,r+1,l+1}}
$$
uniformly in $|\sigma|>\sigma_0$.

The analogous conclusion also holds with
\begin{equation*}
\tilde
r+\ell+1/2-\im\alpha_-(\sigma)<0,\quad\ell+1/2-\im\alpha_+(\sigma)>0,\quad
r+1/2-\im\alpha_-(\sigma)<0
\end{equation*}
and $\im\sigma\leq 0$.
\end{thm}

\begin{rem}
Note that the estimates in Theorem~\ref{thm:high} have a loss of
$|\sigma|^{-1}$ relative to elliptic large-parameter estimates that
hold for $P(\sigma)$ when $\sigma$ is in a cone bounded away from the
real axis: the latter correspond to $P(\sigma):\Hsch^{s,r}\to\Hsch^{s-2,r}$. This is due to the fact that in the more precise
function spaces used in this statement $\hat P(\sigma)$ is not elliptic.
\end{rem}

The structure of this paper is the following. In
Section~\ref{sec:pseudo} we recall the necessary background for
pseudodifferential operator algebras. In Section~\ref{sec:operator} we
discuss in detail the assumptions on $P(\sigma)$, and the form of the
conjugate $\hat P(\sigma)$, as well as elliptic estimates. In
Section~\ref{sec:commutator} we then provide the positive commutator estimates
that prove Theorem~\ref{thm:main}. Finally in Section~\ref{sec:high}
we prove the high energy version, Theorem~\ref{thm:high}.

I am very grateful for numerous discussions with Peter Hintz, various
projects with whom have formed the basic motivation for this work. I
also thank Dietrich H\"afner and  Jared Wunsch for their interest in
this work which helped to push it towards completion, and Jesse
Gell-Redman for comments improving it.

\section{Pseudodifferential operator algebras}\label{sec:pseudo}
Three operator algebras play a key role in this paper on the manifold
with boundary $X$. Below we use $x$ as a boundary defining function,
and $y_j$, $j=1,\ldots,n-1$, as local coordinates on $\pa X$, extended to a collar
neighborhood of the boundary. We also use the convention that vector
fields and differential operators, of various classes discussed below,
have smooth, i.e.\ $\CI(X)$, coefficients unless otherwise
indicated. The notation for symbolic coefficients of order $l$ is
$S^{l}\Diff(X)$, where $\Diff$ obtains subscripts according to the
algebra being studied. Here recall that symbols, or conormal
functions, of order $l$, are $\CI(X^\circ)$ functions which are
bounded by $C_0x^{-l}$, and for which iterated application of vector
fields tangent to the boundary $\pa X$, i.e.\ elements of $\Vb(X)$, results in
a similar (with different constants) bound. In
local coordinates, elements of $\Vb(X)$ are linear combinations of
$x\pa_x$ and $\pa_{y_j}$, so the contrast between $\CI$ and $S^0$
coefficients is regularity with respect to $\pa_x$ vs.\ $x\pa_x$. Classical symbols are those with a one-step polyhomogeneous asymptotic
expansion at $\pa X$; thus, classical elements of $S^0$ are exactly
elements of $\CI$.

The first algebra that plays a role is Melrose's scattering algebra, \cite{RBMSpec}, $\Psisc^{s,r}(X)$; the
spectral family of the Laplacian of a scattering metric lies in
$\Psisc^{2,0}(X)$. This algebra is based on the Lie algebra of scattering
vector fields $\Vsc(X)=x\Vb(X)$, where we recall that $\Vb(X)$ is the Lie algebra of
b-vector fields, i.e.\ vector fields tangent to $\pa X$, and the
corresponding algebra $\Diffsc(X)$ consisting of finite sums of finite
products of scattering vector fields and elements of $\CI(X)$. In
local coordinates as above, elements of $\Vsc(X)$ are linear
combinations of $x^2\pa_x,x\pa_{y_j}$. These
vector fields are all smooth sections of the vector bundle $\Tsc X$,
with local basis $x^2\pa_x,x\pa_{y_j}$,
and thus their principal symbols are exactly smooth (in the base
point) fiber-linear functions on the dual bundle $\Tsc^*X$ (with local
basis $x^{-2}\,dx,x^{-1}\,dy_j$, the coefficients, which give fiber
coordinates, are denoted by $\tau$ and $\mu_j$, i.e.\  a covector is
of the form $\tau (x^{-2}\,dx)+\sum_j\mu_j(x^{-1}\,dy_j)$); the differential operators
have thus principal symbols which are fiber-polynomials. In order to
familiarize ourselves with this, we note that if $X$ is the radial
compactification $\overline{\RR^n}$ of $\RR^n_z$, i.e.\ a sphere at infinity is added, so
that the result is a closed ball, with $x=r^{-1},y_j$ being local
coordinates near a point on the boundary, where $r$ is the Euclidean
radius function and $y_j$ are local coordinates on the sphere, then
$\Vsc(X)$ is exactly the collection of vector fields of the form $\sum
a_j\pa_{z_j}$, where $a_j$ are smooth on $X$. Correspondingly, in this
case, $\Tsc^*X$ is naturally identified with $\overline{\RR^n_z}\times
(\RR^n)_\zeta^*$, i.e.\ (a partially compactified version of) the most familiar phase space in microlocal
analysis. The class of pseudodifferential operators $\Psisc^{s,r}(X)$ in this case,
going back to Parenti and Shubin
\cite{Parenti:Operatori,Shubin:Pseudo}, is standard quantizations of
symbols $a\in S^{s,r}(\Tsc^*X)$ on $\Tsc^*X$ of orders $(s,r)$, where $s$ is the differential
and $r$ is the decay order:
$$
|D_z^\alpha D_\zeta^\beta a(z,\zeta)|\leq
C_{\alpha\beta}\langle z\rangle^{r-|\alpha|}\langle \zeta\rangle^{s-|\beta|}.
$$
The phase
space in general for $\Psisc^{s,r}(X)$ is thus $\Tsc^*X$,
quantization maps can be realized by using a partition of unity within
coordinate charts each of which is either disjoint from the boundary
or is of the form as above, i.e.\  a coordinate chart on the sphere
times $[0,\ep)_x$, which in turn can be identified with an
asymptotically conic region at
infinity in Euclidean space, so the $\RR^n$-quantization can be
used. (One also adds general Schwartz kernels which are Schwartz on
$X\times X$, i.e.\ are in $\dCI(X\times X)$.) The principal symbols in
this algebra are taken modulo lower order terms in terms of both orders,
i.e.\ in
$$
S^{s,r}(\Tsc^*X)/S^{s-1,r-1}(\Tsc^*X)=S^{s,r}(\overline{\Tsc^*}X)/S^{s-1,r-1}(\overline{\Tsc^*}X),
$$
where $\overline{\Tsc^*}X$ denotes the fiber radial compactification
of $\Tsc^*X$. In particular, vanishing of this principal symbol
captures relative compactness on $L^2_\scl$-based Sobolev spaces; here
$L^2_\scl$ is the $L^2$-space with respect to any Riemannian sc-metric
(i.e.\ a smooth positive definite inner product on $\Tsc X$), which is
the standard $L^2$ space on $\RR^n$ in case $X=\overline{\RR^n}$.

The second algebra is Melrose's b-algebra \cite{Melrose:Atiyah}, whose Lie
algebra of vector fields, $\Vb(X)$, has already been discussed. In
local coordinates, elements of the latter are linear combinations of
$x\pa_x$ and $\pa_{y_j}$, so again are all smooth sections of a vector
bundle, $\Tb X$, with local basis $x\pa_x$ and $\pa_{y_j}$, and thus
their principal symbols are smooth fiber-linear functions on the dual bundle
$\Tb^*X$, with local basis $x^{-1}\,dx$ and $dy_j$ (with coefficients
denoted by $\taub$ and $(\mub)_j$, so covectors are written as $\taub\,(x^{-1}\,dx)+\sum_j(\mub)_j\,dy_j$). The corresponding
pseudodifferential algebra $\Psib^{\tilde r,l}(X)$, with $\tilde r$
the differential, $l$ the decay, order, which Melrose defined via
describing their Schwartz kernels on a resolved space, called the
b-double space, is
closely related to H\"ormander's uniform algebra $\Psi_\infty^{\tilde
  r}(\RR^n)$ \cite[Chapter~18.1]{Hor}. Namely, using $t=-\log x$ we are working in a cylinder
$[T,\infty)_t\times U$, $U$ a coordinate chart on $\pa X$, and for
instance Schwartz kernels of elements of $x^{-l} \Psi_\infty^{\tilde
  r}(\RR^n)=e^{lt}\Psi_\infty^{\tilde
  r}(\RR^n)$ which have support (with prime denoting the right,
unprime the left, factor on the product space) in $|t-t'|<R$ are
elements of $\Psib^{\tilde r,l}(X)$ and indeed capture (locally)
$\Psib^{\tilde r,l}(X)$ modulo smoothing operators,
$\Psib^{-\infty,l}(X)$. In general one adds smooth Schwartz
kernels which are superexponentially decaying in $|t-t'|$, as well as
Schwartz kernels relating to disjoint coordinate charts on $\pa X$
with similar decay, see \cite[Section~6]{Vasy:Minicourse} for a more
thorough description from this perspective. In this
algebra the principal symbol map captures only the behavior at
fiber infinity, i.e.\ in the differential order sense, and takes
values in
$$
S^{\tilde r,l}(\Tb^*X)/S^{\tilde r-1,l}(\Tb^*X)=S^{\tilde
  r,l}(\overline{\Tb^*}X)/S^{\tilde r-1,l}(\overline{\Tb^*}X).
$$
This
principal symbol is a $*$-algebra homomorphism, so
$$
\sigma_{\tilde r+\tilde r',l+l'}(AA')=\sigma_{\tilde
  r,l}(A)\sigma_{\tilde r',l'}(A'),\qquad A\in\Psib^{\tilde r,l}(X),\ A'\in\Psib^{\tilde r',l'}(X),
$$
so the algebra is commutative to leading order in the differential
sense, i.e.
$$
[A,A']\in\Psib^{\tilde r+\tilde r'-1,l+l'}(X),
$$
but there is no gain in decay. The principal symbol of the commutator
as an element of $\Psib^{\tilde r+\tilde r'-1,l+l'}(X)$ is given by the
usual Hamilton vector field expression:
$$
\sigma_{m+m'-1,l+l'}([A,A'])=\frac{1}{i}\sH_a a',\ a=\sigma_m(A),\ a'=\sigma_{m'}(A').
$$
For $l=0$, $\sH_a$ is a b-vector field on $\Tb^*X$, i.e.\ is tangent
to $\Tb^*_{\pa X}X$ (and in general it
simply has an extra weight factor); indeed in local coordinates it
takes the form
\begin{equation}\begin{aligned}\label{eq:b-Ham-vf}
&(\pa_{\taub} a) (x\pa_x)-(x\pa_x a) \pa_{\taub}+\sum 
_j\big((\pa_{(\mub)_j}a)\pa_{y_j}-(\pa_{y_j}a)\pa_{(\mub)_j}\big) \\
&\qquad=(-\pa_{\taub} a) \pa_t-\pa_t a (-\pa_{\taub})+\sum 
_j\big((\pa_{(\mub)_j}a)\pa_{y_j}-(\pa_{y_j}a)\pa_{(\mub)_j}\big),
\end{aligned}\end{equation}
where the $-$ signs in the $\pa_t$-version correspond to
$\taub\,\frac{dx}{x}=-\taub\,dt$; notice that the second line is the
standard form of the Hamilton vector field taking into account that
$\taub$ is the {\em negative} of the canonical dual coordinate of $t$.

Principal symbol based constructions and considerations (ellipticity,
propagation of singularities, etc.) do {\em not} give rise to relatively compact errors on
$L^2_\bl$-based Sobolev spaces; here $L^2_\bl$ is the $L^2$-space with
respect to any Riemannian b-metric
(i.e.\ a smooth positive definite inner product on $\Tb X$), which in
the cylindrical picture above is simply the standard $L^2$ space on
the cylinder. However, in addition there is a normal operator, which
captures the behavior of an element of $\Psib^{\tilde r,l}(X)$ at
$X$. For differential operators, $P\in S^0\Diffb^{\tilde r,l}(X)$,
which is at least to leading order at the boundary is
smooth (which in the cylindrical picture means that the coefficients have a limit as
$t\to+\infty$, with exponential convergence to the limit), this
amounts to restricting the coefficients of $x^l$ times the operator to the boundary and obtaining
a model operator on $[0,\infty)_x\times\pa X$ which is dilation
invariant in $x$ (which amounts to translation invariance in $t$ on
$\RR_t\times\pa X$);
there is an analogous statement for pseudodifferential operators. If
an operator is elliptic in the principal symbol sense, and its normal
operator is invertible on a weighted Sobolev space, then the original
operator is Fredholm between correspondingly weighted b-Sobolev spaces
(shifted by the decay order $l$ we factored out).

There is a common resolution of these two algebras in the form of the
third relevant algebra, which is the second microlocalized, at the
zero section, scattering algebra, $\Psiscb^{s,r,l}(X)$, which is
described in more detail in this context in
\cite[Section~5]{Vasy:Zero-energy}. Here the
symbol space $S^{s,r,l}$ can be arrived at in two different ways. From the second
microlocalization perspective, one takes $\overline{\Tsc^*}X$, and
blows up the zero section over the boundary $o_{\pa X}$. The new front
face is naturally identified with $\Tb^*_{\pa X} X$. In this
perspective, one is looking at scattering pseudodifferential operators
which are singular at the zero section.
Now the three orders of $\Psiscb^{s,r,l}(X)$, and correspondingly of $S^{s,r,l}$, are the 
sc-differential order $s$, the sc-decay order $r$ and the b-decay order $l$
respectively, i.e.\ they are the symbolic orders of amplitudes used
for the quantization at the three hypersurfaces: sc-fiber infinity,
the lift of $\Tsc^*_{\pa X} X$, and the new front face. Thus, we adopt a second microlocalization-centric 
approach in the order convention, see Figure~\ref{fig:2-micro}.

\begin{figure}[ht]
\begin{center}
\includegraphics[width=120mm]{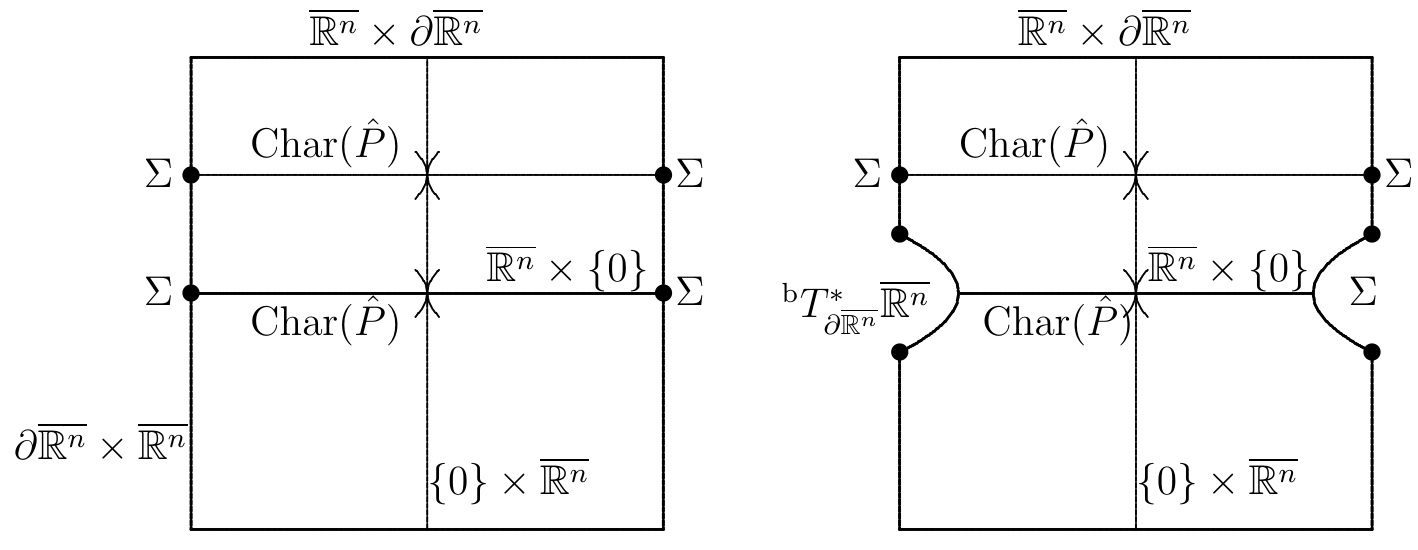}
\end{center}
\caption{Second microlocalized Euclidean space $\RR^{n}$. The left
  hand side is the fiber-compactified sc-cotangent bundle,
  $\overline{\Tsc^*}\overline{\RR^n}=\overline{\RR^n}\times\overline{(\RR^n)^*}$,
  the right hand side is its blow-up at the boundary of the zero
  section. The
  (interior of the) front
  face of the blow-up, shown by the curved arcs, can be identified
  with $\Tb^*_{\pa\overline{\RR^n}}\overline{\RR^n}$. The
  characteristic set of $\hat P(\sigma)$, $\sigma\neq 0$, discussed in
  Section~\ref{sec:operator}, is also shown,
  both from the compactified perspective, as $\Sigma$, which is a
  subset of the boundary, and from the conic perspective, here conic
  in the base (i.e.\ the dilations are in the $\RR^n_z$, spatial, factor), as
  $\Char(\hat P)$. On the second microlocal figure on the right, the
  characteristic set within the boundary lies at the lift of the
  fibers of the sc-cotangent bundle over the boundary; from the
  b-perspective, it thus corresponds to symbolic behavior, and lies at
  fiber infinity. The fiber of cotangent bundle over the
  origin, i.e.\ $\{0\}\times\overline{(\RR^n)^*}$, is also indicated;
  this is only special from the conic (dilation) perspective, in which
  it is the analogue of the zero section in standard microlocal analysis.}
\label{fig:2-micro}
\end{figure}

On the other hand, from an
analytically better behaved, but geometrically equivalent,
perspective, one takes $\overline{\Tb^*}X$, and blows up the corner,
namely fiber infinity at $\pa X$. The new front face is then
$\overline{\Tsc^*_{\pa X}}X$, blown up at the zero section, see Figure~\ref{fig:b-2-micro}. These two resolved
spaces are naturally the same, see \cite[Section~5]{Vasy:Zero-energy},
in the sense that the identity map in
the interior (as both are identified with $\overline{T^*}X^\circ$
there) extends smoothly to the boundary; this can be checked easily by
noting that
$$
\tau\,(x^{-2}\,dx)+\sum_j\mu_j(x^{-1}\,dy_j)=\taub(x^{-1}\,dx)+\sum_j(\mub)_j\,dy_j
$$
shows
$$
\tau=x\taub,\ \mu=x\mub.
$$
The advantage of the
b-perspective is that the b-quantization, etc., procedures work
without a change, since the space of conormal functions, i.e.\
symbols, is unchanged under blowing up a corner. Moreover, it allows
to capture global phenomena at the Lagrangian, and thus compactness
properties, unlike the usual second microlocal perspective in which
Lagrangianizing errors are treated as residual. In particular, we have
$$
\Psiscb^{\tilde r,\tilde r+l,l}(X)=\Psib^{\tilde r,l}(X).
$$
The algebra $\Psiscb(X)$ combines the features of the previous two
algebras, thus the principal symbol is in $S^{s,r,l}/S^{s-1,r-1,l}$,
does not capture relative compactness, and there is a normal operator,
which when combined with the principal symbol, does capture relative
compactness and thus Fredholm properties. Because of the
aforementioned identification, one can consider the sc-decay part of
the principal symbol to be described by a function on
$[\overline{\Tsc^*_{\pa X}}X;o_{\pa X}]$, up to overall weight factors,
at least if the pseudodifferential operator is to leading order (in
sc-decay) classical.

\begin{figure}[ht]
\begin{center}
\includegraphics[width=120mm]{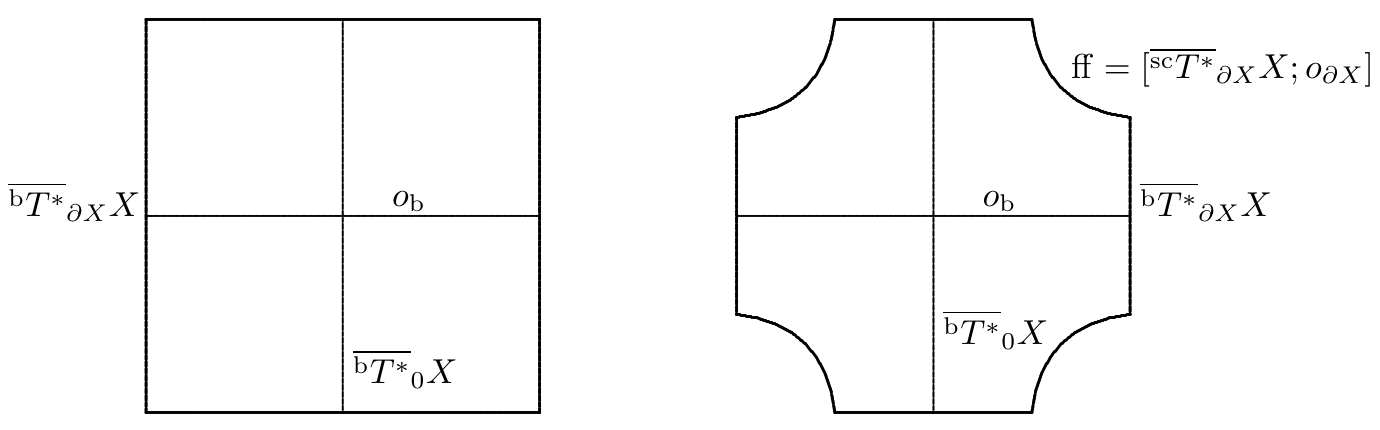}
\end{center}
\caption{The second microlocal space, on the right, obtained by blowing up the corner
  of $\overline{\Tb^*}X$, shown on the left.}
\label{fig:b-2-micro}
\end{figure}

In all these cases one has corresponding Sobolev spaces, namely
$\Hsc^{s,r},\Hb^{\tilde r,l},\Hscb^{s,r,l}$, all of which are subspaces of
tempered distributions $u$, i.e.\ the dual space of $\dCI(X)$, i.e.\ $\CI$ functions
vanishing to infinite order at $\pa X$. Of these, $\Hsc^{s,r}$ is
locally, in the sense of asymptotically conic regions discussed
earlier in this section, the standard weighted Sobolev space on $\RR^n$, $\langle
z\rangle^{-r} H^s(\RR^n)$. Alternatively, for $s\geq 0$ one can simply take an
elliptic element $A$ of $\Psisc^{s,0}(X)$, elliptic in the sense of
the sc-differential order (i.e.\ as $|\zeta|\to\infty$ in the local model), and define the space as
$u\in x^r L^2$, $L^2=L^2_\scl(X)$, for which $Au\in x^r L^2$.
Here the
choice of the elliptic element is irrelevant, and all such elliptic
elements $A$ give rise to equivalent squared norms:
\begin{equation}\label{eq:sc-squared-norm}
\|Au\|^2_{L^2}+\| x^{-r}u\|^2_{L^2}.
\end{equation}
Similarly, in the cylindrical identification discussed earlier, $\Hb^{\tilde r,l}$ is
locally the weighted Sobolev space $e^{(l-n/2)t} H^{\tilde r}(\RR^n)$, where
the distribution should be supported in $(T,\infty)_t\times U$, the
type of region
discussed before. Here the exponent $-nt/2$ enters in our definition so that the $l=0$
space is $L^2(X)=L^2_\scl(X)$: the density is a positive
non-degenerate multiple of
$$
x^{-n-1}\,|dx\,dy|=e^{nt}\,|dt\,dy|.
$$
{\em This is a shift by
  $e^{-nt/2}=x^{n/2}$ relative to the usual convention for b-Sobolev
  spaces, see e.g.\ \cite{Melrose:Atiyah} or
  \cite[Section~5.6]{Vasy:Minicourse}, and is made, as in \cite{Vasy:Zero-energy}, so that the base spaces,
  corresponding to all orders being $0$, are the same $L^2$-space for all Sobolev
  scales we consider.} For $\tilde r\geq 0$ this again amounts to having an elliptic element $A$
of
$\Psib^{\tilde r,0}(X)$, elliptic in the symbolic sense (i.e.\ in the
usual sense for H\"ormander's uniform algebra in the local model
discussed earlier), mapping the distribution $u$ to $x^l L^2$,
with norm
\begin{equation}\label{eq:b-squared-norm}
\|Au\|^2_{L^2}+\| x^{-l}u\|^2_{L^2}.
\end{equation}
The second microlocal spaces are refinements of $\Hb^{\tilde r,l}$;
choosing any $\tilde r\leq\min(s,r-l)$, $\Hscb^{s,r,l}(X)$ is the
subspace of $\Hb^{\tilde r,l}$ for which there exists an elliptic
element $A$ of $\Psiscb^{s,r,0}(X)$, elliptic in the standard symbolic
sense corresponding to the first two orders, for which $Au\in x^l L^2$ with the squared
norm
\begin{equation}\label{eq:scb-squared-norm}
\|x^{-l}Au\|^2_{L^2}+\|u\|^2_{\Hb^{\tilde r,l}},
\end{equation}
and the second term can be replaced by
$\|u\|^2_{\Hb^{0,l}}=\|x^{-l}u\|^2_{L^2}$ if $\tilde r$ can be taken
to be $\geq 0$. We reiterate that we
use the scattering $L^2$ space as the base space for our normalization
of orders in all cases, so when
all their indices are $0$, all these spaces are simply
$L^2=L^2_\scl(X)$.

In the high energy setting we also need the semiclassical version of all these algebras. In
these one adds a parameter $h\in(0,1]$; for fixed $h>0$ the constant
$h$ has no significant effect, so the main point is the uniform
behavior as $h\to 0$. In both the sc- and b-settings, the vector
fields that generate the semiclassical differential operator algebras, $\Diffsch(X)$,
resp.\ $\Diffbh(X)$, over $\CI([0,1]_h\times X)$, are $h$ times the
standard vector fields, i.e.\ $h\Vsc(X)$, resp.\ $h\Vb(X)$. Thus, for
instance, semiclassical scattering differential operators are built
from $hx^2D_x$ and $hxD_{y_j}$ in local coordinates. There are then
semiclassical pseudodifferential algebras in both of these cases. Much
as $\Psisc(X)$ is one of the standard pseudodifferential algebras when
$X=\overline{\RR^n}$, $\Psisch(X)$ is one of the standard
semiclassical pseudodifferential algebras in this case; elements are
semiclassical quantizations
$$
(A_h u)(z,h)=(2\pi h)^{-n}\int e^{i\zeta\cdot (z-z')/h}a(z,\zeta,h)u(z')\,dz'\,d\zeta
$$
of
symbols $a\in \CI([0,1);S^{s,r}(\Tsc^*X))$ on $\Tsc^*X$ of orders $(s,r)$, where $s$ is the differential
and $r$ is the decay order:
$$
|D_h^jD_z^\alpha D_\zeta^\beta a(z,\zeta,h)|\leq
C_{\alpha\beta j}\langle z\rangle^{r-|\alpha|}\langle \zeta\rangle^{s-|\beta|};
$$
here one can simply demand boundedness in $h$ (not in its derivatives)
instead. The phase space is then $\overline{\Tsc^*}X\times[0,1)_h$,
and the principal symbol is understood modulo additional decay in $h$, i.e.\ in
$$
S^{s,r}(\overline{\Tsc^*}X\times[0,1)_h)/hS^{s-1,r-1}(\overline{\Tsc^*}X\times[0,1)_h),
$$
so there is a new, semiclassical, principal symbol, given by the
restriction of $a$ to $h=0$. Since the localization becomes stronger
as $h\to 0$, one can transfer this algebra to manifolds with boundary
just as we did for $\Psisc(X)$. We refer to
\cite{Vasy-Zworski:Semiclassical} for more details, and
\cite{Zworski:Semiclassical} for a general discussion of semiclassical
microlocal analysis.

The b-version is completely similar, locally (using the logarithmic
identification above) based on the
semiclassical quantization of H\"ormander's
uniform algebra, i.e.\ symbols in $\CI([0,1)_h;S_\infty^{\tilde
  r}(\Tb^*X))$:
$$
|D_h^jD_{\tilde z}^\alpha D_{\tilde \zeta}^\beta a(\tilde z,\tilde \zeta,h)|\leq
C_{\alpha\beta j}\langle \tilde\zeta\rangle^{\tilde r-|\beta|},
$$
where $\tilde z=(t,y)=(-\log x,y)$ locally. In particular, principal
symbols are in
$$
S^{\tilde r,l}(\overline{\Tb^*}X\times[0,1)_h)/hS^{\tilde r-1,l}(\overline{\Tb^*}X\times[0,1)_h),
$$
and again there is a normal operator. We refer to
\cite[Appendix~A.3]{Hintz-Vasy:Stability} for more details.

Finally the second microlocalized at the zero section algebra arises, as
before, by blowing up the zero section at $\pa X\times[0,1)_h$ in
$\overline{\Tsc^*}X\times[0,1)_h$, though it is better to consider it
from the b-perspective, blowing up the corner $\pa_2\overline{\Tb^*X}$
of $\Tb^*X$, times $[0,1)_h$, in $\overline{\Tb^*X}\times[0,1)_h$. Here
$[0,1)_h$ is a parameter for both perspectives, namely it is a factor
both in the space within which the blow-up is taking place and in the submanifold being blown up, so the resulting space
is
$$
[\overline{\Tb^*
  X};\pa_2\overline{\Tb^*X}]\times[0,1)_h=[\overline{\Tsc^*}X;o_{\pa
  X}]\times[0,1)_h,
$$
i.e.\ the symbols are smooth functions of $h$ with values in the
non-semiclassical second microlocal space. Since this is a blow up of
the codimension 2 corner of $\overline{\Tb^*X}\times[0,1)_h$ in the
first factor, much as in the
non-semiclassical setting, one can use the usual (now semiclassical)
b-pseudodifferential algebra for quantizations, properties, etc.

The semiclassical Sobolev spaces are the standard Sobolev spaces, but
with an $h$-dependent norm. Thus, on $\RR^n$, these are defined using
the semiclassical Fourier transform
$$
(\cF_h u)(\zeta,h)=(2\pi h)^{-n}\int_{\RR^n} e^{-iz\cdot\zeta/h} u(z,h)\,dz;
$$
so that
$$
\|u\|_{H^s_\semi(\RR^n)}=\|\langle\zeta\rangle^s \cF_h u\|_{L^2(\RR^n)},
$$
while
$$
\|u\|_{H^{s,r}_\semi}=\|\langle z\rangle^{r} u\|_{H^s_\semi},
$$
and then the definition of $\Hsch^{s,r}(X)$ locally reduces to
this. The Sobolev spaces for the other operator algebras are analogous.
Thus, \eqref{eq:sc-squared-norm},
\eqref{eq:b-squared-norm}, \eqref{eq:scb-squared-norm} are replaced by
an equation of the same form but with
$A\in\Psisch^{s,0}(X)$, resp.\ $A\in\Psibh^{\tilde r,0}(X)$, resp.\
$A\in\Psiscbh^{s,r,0}(X)$, elliptic in the relevant symbolic senses.

\section{The operator}\label{sec:operator}
We first define the class of operator families we consider, drawing
comparisons with \cite{Vasy:Zero-energy}, where the $\sigma\to 0$
limit was analyzed in the
unconjugated framework.
Thus, we let $g$ be a scattering metric,
$$
g-g_0\in S^{-1}(X;\Tsc^*X\otimes_s\Tsc^*X),\
g_0=x^{-4}\,dx^2+x^{-2}g_{\pa X},
$$
$g_{\pa X}$ a metric on $\pa X$, so $g$ is asymptotic to a conic metric
$g_0$ on $(0,\infty)\times\pa X$, and we indeed make the assumption
that $g-g_0$ even has a leading term, i.e.\ for some $\delta>0$,
$$
g-g_0\in x\CI (X;\Tsc^*X\otimes_s\Tsc^*X)+S^{-1-\delta}(X;\Tsc^*X\otimes_s\Tsc^*X)
$$
These are stronger requirements than in \cite{Vasy:Zero-energy}, where
$g-g_0\in S^{-\delta}(X;\Tsc^*X\otimes_s\Tsc^*X)$ was allowed,
$\delta>0$, but this is due to our desire to obtain a more precise
conclusion, albeit in a non-zero energy regime. In fact, these
requirements can be relaxed, as only the normal-normal component of
$g-g_0$ actually needs to have such an asymptotic behavior, but we
do not comment on this further.

In \cite{Vasy:Zero-energy}, due to the near zero energy regime being
considered, we worked in a b-framework from the beginning. For
non-zero energies weaker (in terms of the operator algebra), scattering, assumptions are natural, though
we impose stronger asymptotic requirements on these.
Then we consider
\begin{equation*}
P(\sigma)=P(0)+\sigma Q-\sigma^2,\qquad P(0)\in S^{0}\Diffsc^2(X),\ Q\in
S^{-1}\Diffsc^1(X),
\end{equation*}
$P(0)$ elliptic,
$$
P(0)-\Delta_g\in S^{-1}\Diffsc^1(X),
$$
Notice that this means that for real $\sigma$,
$P(\sigma)-P(\sigma)^*\in S^{-1}\Diffsc^1(X)$.

We in fact make the stronger assumption that $P(0)-\Delta_g,Q$ have
leading terms:
\begin{equation}\begin{aligned}\label{eq:PQ-stronger-sc}
P(0)-\Delta_g&\in x\Diffsc^1(X)+S^{-1-\delta}\Diffsc^1(X),\\
Q&\in x\Diffsc^1(X)+S^{-1-\delta}\Diffsc^1(X);
\end{aligned}\end{equation}
thus
$$
\sigma\in\RR\Rightarrow P(\sigma)-P(\sigma)^*\in x\Diffsc^1(X)+S^{-1-\delta}\Diffsc^1(X).
$$

Here in fact $Q$ can have arbitrary smooth dependence on
$\sigma$ if one stays sufficiently close to real values of $\sigma$,
in particular for real $\sigma$. Note that for fixed real $\sigma$, $Q$ can be incorporated
into $P(0)$.
While the restrictions we imposed can be relaxed, in that leading
terms are only required in some particular components, essentially
amounting to the radial set that is moved to the zero section, in this paper
we keep the assumption of this form.
Note that
$$
x^3\Diffb^2(X)+S^{-3-\delta}\Diffb^2(X)\subset x\Diffsc^2(X)+S^{-1-\delta}\Diffsc^2(X)
$$
and
$$
x^2\Diffb^1(X)+S^{-2-\delta}\Diffb^1(X)\subset x\Diffsc^1(X)+S^{-1-\delta}\Diffsc^1(X);
$$
the expressions on the left hand side correspond to the `category' of
operators used in \cite{Vasy:Zero-energy}, in so far as b-spaces are used, although here we have stronger
decay assumptions.

We mention that
\begin{equation}\label{eq:Delta-sc-def}
\Delta_{g_0}=\Delta_{\scl}=x^{n+1} D_x x^{-n-1} x^4D_x+x^2\Delta_{\pa X}
\end{equation}
is the model scattering Laplacian at infinity.

From the Lagrangian perspective we consider a conjugated version of $P(\sigma)$. Thus,
let
$$
\hat P(\sigma)=e^{-i\sigma/x}P(\sigma)e^{i\sigma/x}.
$$
Since conjugation by $e^{i\sigma/x}$ is well-behaved in the
scattering, but not in the b-sense, it is actually advantageous to
first perform the conjugation in the scattering setting, and then
convert the result to a b-form.
At the scattering principal level, the effect of the conjugation is to
replace $\tau$ by $\tau-\sigma$ and leave $\mu$ unchanged,
corresponding to
$$
e^{-i\sigma/x}(x^2D_x)e^{i\sigma/x}=x^2D_x-\sigma,\qquad
e^{-i\sigma/x}(xD_{y_j})e^{i\sigma/x}=xD_{y_j}.
$$
Since the
principal symbol of $P(\sigma)$ in the
scattering decay sense, so at $x=0$, is
$$
p(\sigma)=\tau^2+\mu^2-\sigma^2,
$$
the
principal symbol $\hat p(\sigma)$ of $\hat P(\sigma)$ is
$$
\hat p(\sigma)=x^2(\taub^2+\mub^2)-2\sigma
x\taub=\tau^2+\mu^2-2\sigma\tau.
$$
Moreover, for $\sigma$ real, if $P(\sigma)$ is formally self-adjoint,
so is $\hat P(\sigma)$; in general
$$
\sigma\in\RR\Rightarrow \hat P(\sigma)-\hat P(\sigma)^*\in x\Diffsc^1(X)+S^{-1-\delta}\Diffsc^1(X).
$$

In order to have a bit more precise description, it is helpful to compute $\hat
P(\sigma)$ somewhat more explicitly.

\begin{prop}\label{prop:hat-PQ-form}
We have
\begin{equation}\label{eq:hat-P-form}
\hat P(\sigma)=\hat P(0)+\sigma\hat Q-2\sigma \Big(x^2D_x+i\frac{n-1}{2}x+x\tilde\alpha_+(\sigma)\Big)
\end{equation}
with
\begin{equation*}\begin{aligned}
\hat P(0)&\in x^2\Diffb^2(X)+S^{-2-\delta}\Diffb^2(X),\\
\hat Q&\in x^2\Diffb^1(X)+S^{-2-\delta}\Diffb^1(X),\\
\tilde\alpha_+(\sigma)&\in \CI(X)+S^{-\delta}(X).
\end{aligned}\end{equation*}
\end{prop}

\begin{rem}\label{rem:subprincipal}
A simple computation shows that if we regard $\hat P(\sigma)$ as an
operator on half-densities, using the metric density to identify
functions and half-densities, then the subprincipal symbol of $\hat
P(\sigma)$ {\em at} the sc-zero section (i.e. regarding $\hat
P(\sigma)$ as an element of $\Diffsc^2(X)$) is $-2\sigma
x\tilde\alpha_+(\sigma)$ modulo $S^{-2}+S^{-1}\tau+S^{-1}\cdot\mu$,
with the $S^{-1}$ terms corresponding to the statement holding at the
zero section.
\end{rem}

\begin{proof}
To start with, in local coordinates, we have
\begin{equation}\begin{aligned}\label{eq:P-0-form}
P(0)=&(1+xa_{00})(x^2D_x)^2+\sum_j xa_{0j} ((x^2D_x)
(xD_{y_j})+(xD_{y_j}) (x^2D_x))\\
&+\sum_{i,j} a_{ij}
(xD_{y_i})(xD_{y_j})\\
&+(i(n-1)+ a_0)x(x^2D_x)+\sum_j xa_j (xD_{y_j})+x a',
\end{aligned}\end{equation}
and
\begin{equation}\begin{aligned}\label{eq:Q-form}
Q=b_0 x (x^2D_x)+\sum_j x b_j (xD_{y_j})+b'x,
\end{aligned}\end{equation}
with $a_{00},a_{0j},a_0,a_j,a',b_0,b_j,b'\in \CI(X)+S^{-\delta}(X)$, $a_{ij}\in
\CI(X)+S^{-1-\delta}(X)$, and with $b_0,b_j,b'$ smoothly depending on $\sigma$. Here $i(n-1)$ is
taken out of the $x(x^2 D_x)$ term of $P(0)$ because this way for a
formally selfadjoint operator
$a_0|_{\pa X}$ is real, cf.\
\eqref{eq:Delta-sc-def}; similarly for a
formally selfadjoint operator
$a_j,a',b_0,b_j,b'$ have real restrictions to $\pa X$. (Note that
$a_{ij}$ is real by standard principal symbol considerations, as that
of $P(\sigma)$ is the dual metric function $G$.)

This gives
\begin{equation*}\begin{aligned}
&e^{-i\sigma/x}P(0)e^{i\sigma/x}\\
=&(1+xa_{00})(x^2D_x-\sigma)^2+\sum_j xa_{0j} ((x^2D_x-\sigma)
(xD_{y_j})+(xD_{y_j}) (x^2D_x-\sigma))\\
&+\sum_{i,j} a_{ij}
(xD_{y_i})(xD_{y_j})+(i(n-1)+a_0)x(x^2D_x-\sigma)+\sum_j xa_j (xD_{y_j})+x a',
\end{aligned}\end{equation*}
and
$$
e^{-i\sigma/x}Qe^{i\sigma/x}=b_0 x (x^2D_x-\sigma)+\sum_j xb_j (xD_{y_j})+b'x,
$$
Combining the terms gives
\begin{equation}\label{eq:hat-P-form-b}
\hat P(\sigma)=\hat P(0)+\sigma\hat Q-2\sigma \Big(x^2D_x+i\frac{n-1}{2}x+\frac{1}{2}x\big(-a_{00}\sigma+a_0+b_0\sigma-\sigma^{-1}a'-b'\big)\Big)
\end{equation}
with
\begin{equation}\begin{aligned}\label{eq:hatPQ-full-expression}
\hat P(0)&=P(0)-xa'\in x^2\Diffb^2(X)+S^{-2-\delta}\Diffb^2(X),\\
\hat Q&=Q-b'x-2xa_{00}(x^2D_x)-2\sum_j xa_{0j}(xD_{y_j})\\
&\qquad\qquad\qquad\qquad\in x^2\Diffb^1(X)+S^{-2-\delta}\Diffb^1(X).
\end{aligned}\end{equation}
\end{proof}

Notice that if the coefficients were smooth, rather than merely
symbolic, $\hat P(\sigma)$ would be in $x\Diffb^2(X)$; with the
actual assumptions in general
$$
\hat P(\sigma)\in
x\Diffb^2(X)+S^{-2-\delta}\Diffb^2(X)+S^{-1-\delta}\Diffb^1(X),
$$
with the only term of
\eqref{eq:hat-P-form} that is not in a faster decaying space
being the last one; this is unlike $P(\sigma)$ which
is merely in $\Diffb^2(X)+S^{-1-\delta}\Diffb^2(X)$ due to the $\sigma^2$ term; this one order
decay improvement plays a key role below. Note also the $\sigma^{-1}$
in front of $a'$ in the last parenthetical expression of \eqref{eq:hat-P-form-b}; this
corresponds to the Laplacian with the Coulomb potential having significantly different
low energy behavior than with a short range potential (or no
potential). On the other hand, long range terms in the higher order
terms make no difference even in that case; indeed, the $a_{00}$
contribution even decays as $\sigma\to 0$.

 We also remark that the principal symbol of $\hat P(0)$ vanishes
 quadratically at the scattering zero section, $\tau=0$, $\mu=0$,
 $x=0$, hence the subprincipal symbol makes sense directly there (without
 taking into account contributions from the principal symbol, working
 with half-densities, etc.),
 and this in turn vanishes. (The same is not true for $P(0)$ due to the
 $xa'$ term.)
Since it
will be helpful when considering non-real $\sigma$ below,
we note positivity properties of $\hat P(0)$ and related structural
properties of $\hat Q$.

\begin{lemma}
The operator $\hat P(0)$ is non-negative modulo terms that are either
sub-sub-principal or subprincipal but with vanishing contribution at
the scattering zero section, in the sense that it
has the form
\begin{equation}\label{eq:hat-P-0-nonnegative}
\hat P(0)=\sum_j T_j^*T_j+\sum_j T_j^* T'_j+\sum_j T^\dagger_j T_j+T''
\end{equation}
where $T_j\in x\Diffb^1(X)+S^{-2-\delta}\Diffb^1(X)$, $T'_j,T^\dagger_j\in
x\CI(X)+S^{-1-\delta}(X)$, $T''\in x^2\CI(X)+S^{-2-\delta}(X)$.
Moreover,
\begin{equation}\label{eq:hat-Q-module-form}
\hat Q=\sum_j T_j^* \tilde T'_j+\sum_j \tilde T^\dagger_j T_j+\tilde T''
\end{equation}
with $\tilde T'_j,\tilde T_j^\dagger\in
x\CI(X)+S^{-1-\delta}(X)$, $\tilde T''\in x^2\CI(X)+S^{-2-\delta}(X)$.
\end{lemma}

\begin{rem}
Technically it would be slightly better to replace $T_j$ by a one-form
valued differential operator as that would remove the need of
discussing coordinate charts, and then the form of $\hat P(0)$ would
be immediate from the definition of the Laplacian, with $T_j$ replaced
by the exterior differential or the covariant derivative.
\end{rem}

\begin{proof}
We work in local coordinates, to which we can reduce by taking $T_j$ to be
cutoff versions of what we presently state, with a union taken over
charts. Then we can take the $T_j$ to be $x^2D_x$ and $xD_{y_j}$; then
the adjoints differ from $x^2D_x$, resp.\ $xD_{y_j}$, by elements of
$x\CI(X)+S^{-1-\delta}(X)$, thus the difference can be absorbed into
$T_j'$, $T_j^\dagger$. The
statements then follow from the coordinate form obtained in the proof
of Proposition~\ref{prop:hat-PQ-form}. Note that the removal of the
terms $xa'$ from $P(0)$ and $xb'$ from $Q$ (they being shifted into $\tilde\alpha_+$) is
important in making the membership statements hold.
\end{proof}

In terms of the local coordinate description of $P(\sigma)$ and $\hat
P(\sigma)$, see \eqref{eq:P-0-form} and \eqref{eq:Q-form}, the normal operator
of $\hat P(\sigma)$ in $x\Diffb^2(X)$, which arises by considering the operator
$x^{-1}\hat P(\sigma)$ and freezing the coefficients at
the boundary, is
\begin{equation}\begin{aligned}\label{eq:actual-normal-op-hat}
N(\hat P(\sigma))
&=-2\sigma \Big(x^2D_x+i\frac{n-1}{2}x +\alpha_+ x\Big),\\
&\qquad\qquad \alpha_+=\alpha_+(\sigma)=\tilde\alpha_+(\sigma)|_{\pa
  X}\\
&\qquad\qquad \qquad\qquad =\frac{1}{2}\big(-a_{00}\sigma+a_0+b_0\sigma-\sigma^{-1}a'-b'\big)|_{\pa X};
\end{aligned}\end{equation}
notice that (for $a'=0$) this
degenerates at $\sigma=0$. Invariantly, see
Remark~\ref{rem:subprincipal}, $-2\sigma x\alpha_+(\sigma)$ is the
subprincipal symbol at the sc-zero section, where the quotient is being taken with $S^{-1-\delta}$
in place of $S^{-2}$.
Note that the normal operator is $x$ times the normal vector
field to the boundary plus a smooth function, which, for $\sigma\neq 0$, corresponds to the
asymptotic behavior of the solutions of $\hat P(\sigma)v\in\dCI(X)$
being
$$
x^{(n-1-2i\alpha_+)/2}\CI(\pa X),
$$
modulo faster decaying terms. Here for formally self-adjoint
$P(\sigma)$ when $\sigma$ is real, the $\alpha_+$ term changes
the asymptotics in an oscillatory way (as $\alpha_+$ is real then) but
not the decay rate, but for complex $\sigma$ the decay rate may also
be affected. This corresponds to the asymptotics
$$
e^{i\sigma/x}x^{(n-1-2i\alpha_+)/2}\CI(\pa X)
$$
for solutions of $P(\sigma)u\in\dCI(X)$ for $\sigma\neq 0$.  This
indicates that we can remove the contribution of $\alpha_+$ to
leading decay order by conjugating the operator by $x^{i\alpha_+}$, as
well as have analogous achievements for the $x^{(n-1)/2}$ part of the
asymptotics, but actually that factor is useful for book-keeping when
using the $L^2$, rather than the $L^2_\bl$ inner product, and we do
not remove this here.

Note also that if we instead conjugated by $e^{-i\sigma/x}$, moving
the other radial point to the zero section, we would obtain the normal
operator
\begin{equation}\begin{aligned}\label{eq:other-normal-op-hat}
&2\sigma \Big(x^2D_x+i\frac{n-1}{2}x +\alpha_- x\Big),\\
&\qquad\qquad \alpha_-=\alpha_-(\sigma)=\frac{1}{2}\big(a_{00}\sigma+a_0+b_0\sigma+\sigma^{-1}a'+b'\big)|_{\pa X}.
\end{aligned}\end{equation}
Again invariantly, cf.\ Remark~\ref{rem:subprincipal}, $2\sigma x\alpha_-(\sigma)$ is the
subprincipal symbol at the sc-zero section, where the quotient is being taken with $S^{-1-\delta}$
in place of $S^{-2}$.

Next, we consider the principal symbol behavior at, as well as near, $\pa X$. While (ignoring the
irrelevant $S^{-2-\delta}\Diffb^2(X)+S^{-1-\delta}\Diffb^1(X)$ terms,
which are irrelevant that they do not affect ellipticity near $\pa X$) $\hat P(\sigma)\in
x\Diffb^2(X)$, it is degenerate at the principal symbol level since it
is actually in
$$
x^2\Diffb^2(X)+\sigma
x\Diffb^1(X)\subset\Psib^{2,-2}(X)+\Psib^{1,-1}(X).
$$
Correspondingly, we consider $\hat P(\sigma)$ as an
element of the second microlocalized scattering pseudodifferential
operators, concretely
$$
\hat P(\sigma)\in\Psiscb^{2,0,-1}(X).
$$
Recall that this
space of operators is formally arrived at by blowing up the zero section of the
scattering cotangent bundle at the boundary, but more usefully (in that
quantizations, etc., make sense still) 
by blowing
up the corner of the fiber-compactified b-cotangent bundle); in this
sense the summands $\Psib^{2,-2}(X)$ and $\Psib^{1,-1}(X)$ have the same (sc-decay) order since on the
front face $x$ and $|(\taub,\mub)|^{-1}$ are comparable.

Note that in $\Psiscb^{2,0,-1}(X)$ the operator $\hat P(\sigma)$ is
elliptic in the sc-differential sense, with principal symbol given by
the dual metric function $G$; it is also elliptic in the sc-decay
sense in a neighborhood of the corner corresponding to
sc-fiber-infinity at the boundary, with now the principal symbol being
$G-2\sigma\tau$, $\tau=x\taub$, the sc-fiber
coordinate. Correspondingly, one has elliptic estimates in this
region:
\begin{equation}\label{eq:sc-fiber-infty}
\|B_1 u\|_{\Hscb^{s,r,l}}\leq C(\|B_3\hat P(\sigma)u\|_{\Hscb^{s-2,r,l+1}}+\|u\|_{\Hscb^{-N,-N,-N}}),
\end{equation}
with the third, b-decay, order actually irrelevant, where $B_1,B_3\in\Psiscb^{0,0,0}(X)=\Psib^{0,0}(X)$, $B_1$
microlocalizes to the aforementioned region (i.e.\ has wave front set
there, understood in the strong sense
that we can consider $B_1$ as a scattering ps.d.o., so this imposes
triviality at the b-front face), as does $B_3$, but $B_3$ is elliptic on the wave front set of
$B_1$. Correspondingly, {\em below we always work microlocalized away
  from sc-fiber infinity}, in which region $\Hscb^{s,r,l}$ is the same as
$\Hb^{r-l,l}$, and $\Psiscb^{s,r,l}(X)$ is the same as $\Psib^{r-l,l}(X)$. Thus, if one so wishes, one can use {\em purely} the
b-pseudodifferential and Sobolev space notation.

On the other hand, the principal symbol of $\hat P(\sigma)$ in the scattering decay sense is
$$
\hat p(\sigma)=x^2(\taub^2+\mub^2)-2\sigma
x\taub=\tau^2+\mu^2-2\sigma\tau,
$$
so there is a non-trivial characteristic set, namely where $\hat
p(\sigma)$ vanishes. Notice that if one wants
to consider $\hat P(\sigma)\in\Psiscb^{2,0,-1}(X)$ and its
principal symbol as a function on $[\Tsc^*_{\pa X}X;o_{\pa X}]$, one
should factor out (corresponding to the order $-1$ in the b-decay
sense) the defining function of the front face, i.e.\ (up to
equivalence) $(\tau^2+\mu^2)^{1/2}$; we mostly do not do this explicitly here. (There is an analogous
phenomenon at fiber infinity, but as we discussed already, the operator
is elliptic there, so this is not a region of great interest.)
Concretely we have
\begin{equation}\begin{aligned}\label{eq:symbol-re-im}
\re\hat p(\sigma) &=x^2(\taub^2+\mub^2)-2(\re\sigma)
x\taub=\tau^2+\mu^2-2(\re\sigma)\tau,\\
\im \hat p(\sigma)&=-2(\im\sigma)x\taub=-2(\im\sigma)\tau.
\end{aligned}\end{equation}
For real (non-zero) $\sigma$ thus the characteristic set is the translated sphere
bundle (with a sphere over each base point in $\pa X$),
\begin{equation}\label{eq:real-char-set}
0=\re\hat p(\sigma) =(\tau-\re\sigma)^2+\mu^2-(\re\sigma)^2.
\end{equation}
For non-real complex $\sigma$ this set is intersected with $\tau=0$,
and thus becomes {\em almost} trivial: one concludes that points in
the characteristic set have $\mu=0$, so points of
non-ellipticity are necessarily at the front face $[\Tsc^*_{\pa X}X;o_{\pa
  X}]$. However, to see the behavior there one actually {\em does}
need to rescale by $(\tau^2+\mu^2)^{1/2}$ to obtain
$$
(\tau^2+\mu^2)^{1/2}-2\sigma\frac{\tau}{(\tau^2+\mu^2)^{1/2}},
$$
which does vanish within the front face, $(\tau^2+\mu^2)^{1/2}=0$,
namely at
$\frac{\tau}{(\tau^2+\mu^2)^{1/2}}=0$, but notice that this vanishing
is {\em simple}.

For real $\sigma$, considering $\Tsc^*X$ (rather than its blow-up), the conjugation of the operator is simply
pullback by a
symplectomorphism at the phase space level, and the Hamilton flow has exactly the same
structure as in the unconjugated case, except translated by the
symplectomorphism. Thus, there are two submanifolds of radial points,
one of which is now the zero section, the other is
$\{\tau=2\re\sigma,\ \mu=0\}$, $\tau$ is monotone decreasing along the flow, so for
$\re\sigma>0$, the non-zero section radial set is a source, for
$\re\sigma<0$ it is a sink. For propagation of singularities
estimates at the radial sets, there is a threshold quantity for the
order of the Sobolev spaces, here the scattering decay order, above
which one has microlocal estimates {\em without} a propagation term
(`estimate for free'),
and below which one can propagate estimates into the radial points
from a punctured neighborhood; see
\cite[Section~2.4]{Vasy-Dyatlov:Microlocal-Kerr} in the standard
microlocal context, and \cite[Section~5.4.7]{Vasy:Minicourse} for a
more general discussion that explicitly includes the scattering setting. The relevant quantities are
$$
x^{-1}H_{\hat p(\sigma)}x=\mp\beta_0 x,
$$
with $\beta_0>0$ at the radial points (so $-$ is for sink, the top
line here and thereafter, $+$ for source),
\begin{equation}\label{eq:subprinc-for-rad-pt}
\sigma_{\scl,*,-1}\Big(\frac{\hat P(\sigma)-\hat
  P(\sigma)^*}{2i}\Big)=\pm\beta_0\tilde\beta_\pm x
\end{equation}
at the radial set (where $*$ in the subscript denotes the irrelevant sc-differential order), and then the threshold value is
$$
r_\pm=-\frac{1}{2}-\tilde\beta_\pm.
$$
In our case,
$$
\beta_0=2\,|\re\sigma|,
$$
while \eqref{eq:subprinc-for-rad-pt} at the radial point moved to the
zero section is $-2\re\sigma(\im\alpha_+(\sigma)) x$, giving for $\re\sigma>0$,
when this is a sink, $\tilde\beta_+=-\im\alpha_+(\sigma)$, while for $\re\sigma<0$,
when this is a source, $\tilde\beta_-=-\im\alpha_+(\sigma)$ again, hence in either case we have a threshold regularity
$$
r_0=-\frac{1}{2} +\im\alpha_+(\sigma).
$$
Similarly, for the radial point not moved to the zero section, the
conjugation corresponding to the reversed sign of $\sigma$ would move
it there, and this conjugation gives \eqref{eq:other-normal-op-hat} as
the normal operator, so we have
$$
r_{\neq 0}=-\frac{1}{2} +\im\alpha_-(\sigma).
$$

The resulting estimate, combining propagation estimates from the radial point
outside the sc-zero section and standard propagation estimates, is
\begin{equation}\label{eq:sc-finite-pts-to-0}
\|B_1 u\|_{\Hscb^{s,r,l}}\leq C(\|B_3\hat P(\sigma)u\|_{\Hscb^{s-2,r+1,l+1}}+\|u\|_{\Hscb^{-N,-N,-N}}),
\end{equation}
with the third, b-decay, order actually irrelevant, $r>r_{\neq 0}$,
and where $B_1,B_3\in\Psiscb^{0,0,0}(X)=\Psib^{0,0}(X)$, $B_1$
microlocalizes away from the zero section (again in the strong sense
that we can consider $B_1$ as a scattering ps.d.o., so this imposes
triviality at the b-front face), as does $B_3$, but $B_3$ is elliptic on the wave front set of
$B_1$, and also on all bicharacteristics in the characteristic set of
$\hat P(\sigma)$ emanating from points in $\WFsc'(B_1)$ towards the
non-zero section radial point, including at these radial points. On the other hand, the estimate that propagates estimates from
a neighborhood of the sc-zero section to the other radial set is
\begin{equation}\label{eq:sc-finite-pts-from-0}
\|B_1 u\|_{\Hscb^{s,r,l}}\leq C(\|B_2 u\|_{\Hscb^{s,r,l}}+
\|B_3\hat P(\sigma)u\|_{\Hscb^{s-2,r+1,l+1}}+\|u\|_{\Hscb^{-N,-N,-N}}),
\end{equation}
with the third, b-decay, order again irrelevant, $r<r_{\neq 0}$,
and
where $B_1,B_2,B_3\in\Psiscb^{0,0,0}(X)=\Psib^{0,0}(X)$, $B_1$
microlocalizes away from the zero section in the same sense as above,
$B_2$ is elliptic on an annular region surrounding the zero section,
$B_3$ elliptic on the wave front set of
$B_1$, and also on all bicharacteristics in the characteristic set of
$\hat P(\sigma)$ emanating from points in $\WFsc'(B_1)$ towards
$\WFsc'(B_2)$, including at the radial points outside the zero section.

Moreover, \eqref{eq:real-char-set} shows
that $\tau$ has the same sign as $\re\sigma$ along the characteristic
set, in the extended sense that it is allowed to become zero. In view of
\eqref{eq:symbol-re-im} thus $\im\hat p(\sigma)$ thus has $-\im\sigma$
times the sign of $\re\sigma$. Correspondingly, the standard complex
absorption estimates, see \cite{Vasy:Minicourse} in the present context, allow propagation of estimates forward along the
Hamilton flow when $\im\sigma\geq 0$ (uniformly as $\im\sigma\to 0$),
and backwards when $\im\sigma\leq 0$, which means for both signs of
$\re\sigma$ that we can propagate estimates {\em towards} the zero
section when $\im\sigma\geq 0$, i.e.\ \eqref{eq:sc-finite-pts-to-0} holds then, and {\em away from} the zero section
when $\im\sigma\leq 0$, i.e.\ \eqref{eq:sc-finite-pts-from-0} holds then; these statements (by standard scattering
results) are valid as long as one stays away from the zero section
itself (where we are using the second microlocal pseudodifferential algebra).

\section{Commutator estimates}\label{sec:commutator}
Since from the standard conjugated scattering picture we already know
that the zero section has radial points, the only operator that can give
positivity microlocally in a symbolic commutator computation is the
weight. Here, in the second microlocal setting at the zero section,
this means two different kinds of weights, corresponding to the
sc-decay (thus microlocally b-differential) and the b-decay orders.  Recall that the actual positive commutator estimates utilize
the computation of
\begin{equation}\label{eq:twisted-comm-expr}
i(\hat P(\sigma)^*A-A\hat P(\sigma))=i(\hat P(\sigma)^*-\hat P(\sigma))A+i[\hat P(\sigma),A]
\end{equation}
with $A=A^*$, so for non-formally-self-adjoint $\hat P(\sigma)$ there is a
contribution from the skew-adjoint part
$$
\im \hat P(\sigma)=\frac{1}{2i}(\hat P(\sigma)-\hat P(\sigma)^*)
$$
of
$\hat P(\sigma)$, most relevant for us when $\sigma$ is not real; here the notation `$\im \hat P(\sigma)$' is motivated by the
fact that its principal symbol is actually $\im \hat p(\sigma)$, with
$\hat p(\sigma)$ being the principal symbol of $\hat P(\sigma)$. It is actually
a bit better to rewrite this, with
$$
\re \hat P(\sigma)=\frac{1}{2}(\hat P(\sigma)+\hat P(\sigma)^*)
$$
denoting the
self-adjoint part of $\hat P(\sigma)$, as
\begin{equation}\label{eq:commutator-expr-8}
i(\hat P(\sigma)^*A-A\hat P(\sigma))=(\im \hat P(\sigma) A+A\im \hat P(\sigma))+i[\re \hat P(\sigma),A].
\end{equation}
If
$A\in\Psib^{2\tilde r-1,2l+1}$, $\hat P(\sigma)\in\Psib^{2,-1}(X)$
implies that the second term is a priori in
$\Psib^{2\tilde r,2l}$. However, it is actually in a smaller space
since $\hat P(\sigma)\in \Psib^{2,-2}(X)+\Psib^{1,-1}(X)$, which in terms
    of the second microlocal algebra means that
    $\hat P(\sigma)\in\Psiscb^{2,0,-1}(X)$. Thus, taking $A\in
    \Psib^{2\tilde r-1,2l+1}=\Psiscb^{2\tilde r-1,2(\tilde
    r+l),2l+1}(X)$, the commutator is in fact in
$\Psiscb^{2\tilde r,2(\tilde r+l)-1,2l}(X)$, so the scattering decay
order is $2(\tilde r+l)-1$, and microlocally near the scattering zero
section (where it will be of interest) it is in $\Psib^{2\tilde r-1,2l}(X)$.
Via the usual quadratic form argument this thus estimates
$u$ in $\Hb^{\tilde r-1/2,l}$ in terms of $\hat P(\sigma)u$ in $\Hb^{\tilde
  r-1/2,l+1}$, assuming non-degeneracy. On the other hand, in the first
term we only have $\im \hat P(\sigma)\in \Psib^{1,-1}(X)=\Psiscb^{1,0,-1}(X)$ when $\sigma\notin\RR$, so the
first term is in $\Psiscb^{2\tilde r,2(\tilde r+l),2l}(X)$, so is the
same order, $2l$, in the b-decay sense, but is actually bigger, order
$2(\tilde r+l)$, in the sc-decay sense, which is the usual situation
when one runs positive commutator arguments with non-real principal
symbols, as we will do in the sc-decay sense.

Now, going back to the issue of the zero section consisting of radial points,
we compute the principal symbol of the second term of \eqref{eq:commutator-expr-8} (which is the only
term when $\sigma$ is real and $P(\sigma)$ is formally self-adjoint)
when $A_0 \in\Psib^{2\tilde r-1,2l+1}$ is microlocally the
weight (as mentioned above, only this can give positivity), i.e.
\begin{equation}\label{eq:symb-A0-def}
a=x^{-2l-1}(\taub^2+\mub^2)^{\tilde
  r-1/2},
\end{equation}
is the principal symbol,
so in the second microlocal
algebra, $A_0\in\Psiscb^{2\tilde r-1,2(\tilde r+l),2l+1}$, i.e.\ the scattering decay
order is $2(\tilde r+l)$.

\begin{lemma}\label{lemma:commutator-sc-version}
The principal symbol $H_{\re\hat p(\sigma)}a$ of $i[\re \hat P(\sigma),A_0]$ in $\Psiscb^{2\tilde
  r,2(\tilde r+l)-1,2l}(X)$ is
\begin{equation}\begin{aligned}\label{eq:commutator-sc-version}
&x^{-2l}(\taub^2+\mub^2)^{\tilde
  r-3/2}\Big(4(\re\sigma)\big((l+\tilde r)
\taub^2+(l+1/2)\mub^2\big)-4x(l+\tilde r)\taub(\taub^2+\mub^2)\Big)\\
&=x^{-2(l+\tilde r)+1}(\tau^2+\mu^2)^{\tilde
  r-3/2}\Big(4(\re\sigma)\big((l+\tilde r) \tau^2+(l+1/2)\mu^2\big)-4(l+\tilde r)\tau(\tau^2+\mu^2)\Big).
\end{aligned}\end{equation}
\end{lemma}

\begin{proof}
It is a bit simpler (and more standard) to
compute Poisson brackets using $\Tb^*X$ rather than $\Tsc^*X$, cf.\
\eqref{eq:b-Ham-vf}, so we
proceed this way, and then we re-express the result in $\Tsc^*X$ afterwards. Since the principal symbol of $\re
\hat P(\sigma)$ is
$$
\re \hat p(\sigma)=x^2(\taub^2+\mub^2)-2x\re\sigma\taub,
$$
we compute
\begin{equation}\begin{aligned}\label{eq:commutator-sc-version-pf-1}
&\{x^2(\taub^2+\mub^2)-2x\re\sigma\taub,x^{-2l-1}(\taub^2+\mub^2)^{\tilde
  r-1/2}\}\\
&\qquad=(2x^2\taub-2x\re\sigma)(-2l-1) x^{-2l-1} (\taub^2+\mub^2)^{\tilde
  r-1/2}\\
&\qquad\qquad-(2x^2(\taub^2+\mub^2) -2x\re\sigma\taub) x^{-2l-1}2(\tilde r-1/2)\taub (\taub^2+\mub^2)^{\tilde
  r-3/2}.
\end{aligned}\end{equation}
Expanding and rearranging, we have
\begin{equation}\begin{aligned}\label{eq:commutator-sc-version-pf-2}
&=4(\re\sigma)(l+1/2) x^{-2l}(\taub^2+\mub^2)^{\tilde
  r-1/2}\\
&\qquad+4(\re\sigma)(\tilde r-1/2)x^{-2l}\taub^2 (\taub^2+\mub^2)^{\tilde
  r-3/2}\\
&\qquad-4(l+1/2)x^{-2l+1}\taub (\taub^2+\mub^2)^{\tilde
  r-1/2}\\
&\qquad-4(\tilde r-1/2)x^{-2l+1}\taub (\taub^2+\mub^2)^{\tilde
  r-1/2}\\
&=x^{-2l}(\taub^2+\mub^2)^{\tilde
  r-3/2}\Big(4(\re\sigma)\big((l+1/2) (\taub^2+\mub^2)+(\tilde
r-1/2)\taub^2\big)\\
&\qquad\qquad\qquad\qquad\qquad\qquad\qquad-4x(l+\tilde r)\taub(\taub^2+\mub^2)\Big)\\
&=x^{-2l}(\taub^2+\mub^2)^{\tilde
  r-3/2}\Big(4(\re\sigma)\big((l+\tilde r) \taub^2+(l+1/2)\mub^2\big)-4x(l+\tilde r)\taub(\taub^2+\mub^2)\Big),
\end{aligned}\end{equation}
giving the left hand side of \eqref{eq:commutator-sc-version} as desired.
Finally, substituting $\tau=x\taub$,
$\mu=x\mub$ yield the right hand side.
\end{proof}

\begin{rem}\label{rem:regularizer-choice}
For future reference, we record the impact of having an additional
regularizer factor, namely replacing $a$ by
$$
a^{(\ep)}=a f_\ep.
$$
The role of this is very much standard in positive
commutator estimates, see \cite[Section~5.4]{Vasy:Minicourse} for a
discussion in a similar form, though is slightly delicate in radial points
estimates as radial points limit the regularizability, see
\cite[Section~5.4.7]{Vasy:Minicourse}, \cite[Proof of Proposition~2.3]{Vasy-Dyatlov:Microlocal-Kerr}, as well as earlier work going
back to \cite{RBMSpec} and including \cite[Theorem~1.4]{Haber-Vasy:Radial}. However, in our
second microlocal setting in fact there is no such limitation
as it is the b-decay order that is microlocally limited (near the
scattering zero section), and we are
{\em not} regularizing in that.

One can take the regularizer of the form
$$
f_\ep(\taub^2+\mub^2),\ f_\ep(s)=(1+\ep s)^{-K/2},
$$
where $K>0$ fixed and $\ep\in[0,1]$, with the interesting
behavior being the $\ep\to 0$ limit. Note that $f_\ep(\taub^2+\mub^2)$
is a symbol of order $-K$ for $\ep>0$, but is only uniformly
bounded in symbols of order $0$, converging to $1$ in symbols of
positive order. Then
$$
sf'_\ep(s)=-\frac{K}{2} \frac{\ep s}{1+\ep s} f_\ep(s),
$$
and $0\leq \frac{\ep s}{1+\ep s}\leq 1$, so in particular
$sf'_\ep(s)/f_\ep(s)$ is bounded. The effect of this is to add an
overall factor of $f_\ep(\taub^2+\mub^2)$ to
\eqref{eq:commutator-sc-version} and
\eqref{eq:commutator-sc-version-pf-1}-\eqref{eq:commutator-sc-version-pf-2}, and replace every occurrence of
$\tilde r$, except those in the exponent, by
\begin{equation}\label{eq:tilde-r-replace}
\tilde r+(\taub^2+\mub^2)
\frac{f'_\ep(\taub^2+\mub^2)}{f_\ep(\taub^2+\mub^2)}=\tilde r-\frac{K}{2}\frac{\ep(\taub^2+\mub^2)}{1+\ep(\taub^2+\mub^2)}.
\end{equation}
\end{rem}

\begin{cor}\label{cor:commutator-sc-version}
The principal symbol of $i[\re \hat P(\sigma),A_0]$ in $\Psiscb^{2\tilde
  r,2(\tilde r+l)-1,2l}(X)$ is a positive elliptic multiple of
$\re\sigma$ in $S^{2\tilde r,2(\tilde r+l)-1,2l}$ on the characteristic
set near the image of the
scattering zero section, i.e.\ the b-face, if $l+1/2>0$, and it is a
negative elliptic multiple there if $l+1/2<0$.
\end{cor}

\begin{rem}
The sign restrictions on $l+1/2$ are exactly the restrictions on the
decay order at the radial point moved to the zero section in the
scattering perspective, i.e.\ if using standard scattering pseudodifferential
operators, for
formally self-adjoint $\hat P(\sigma)$, as discussed at the end of
Section~\ref{sec:operator} (cf.\ $r_0$ there).

Note also that adding a regularizer factor as in
Remark~\ref{rem:regularizer-choice} leaves the conclusion valid with
$\re\sigma$ replaced by $(\re\sigma) f_\ep(\taub^2+\mub^2)$, and the
ellipticity uniform in $\ep\in[0,1]$, with the point being that any
appearance of $\tilde r$ in \eqref{eq:commutator-sc-version} (apart
from those in the exponent), that is thus
replaced by \eqref{eq:tilde-r-replace}, comes with an additional
vanishing factor at the zero section (via $\tau$ or $\taub$) and thus
is lower order in the b-decay sense.
\end{rem}

\begin{proof}
On the characteristic set of
$\hat P(\sigma)$, where thus $\re \hat p(\sigma)=0$, we have
$$
0=\re \hat p(\sigma)=(\tau-\re\sigma)^2+\mu^2-(\re\sigma)^2,
$$
so $|\tau-\re\sigma|\leq|\re\sigma|$, and thus $\tau$ has the same sign
as $\re\sigma$, but only in an indefinite sense (thus it may
vanish). Restricted to $\re \hat p(\sigma)=0$, $\mu$ has a
simple zero at the zero section while $\tau$ vanishes quadratically
since at $\tau=0,\mu=0$, $dp$ is $-2(\re\sigma)\,d\tau$, i.e.\ $\tau$
is equal to $\re \hat p(\sigma)$ up to quadratic errors (while $d\mu$ is
linearly independent of this). Correspondingly,  on the right hand
side of \eqref{eq:commutator-sc-version}, not only is the second
term of the big parentheses smaller
than the first near the zero section on account of the extra $\tau$
vanishing factor, but even the $\tau^2$ term is negligible compared to
the $\mu^2$ term, provided that the latter has a
non-degenerate coefficient, i.e.\ provided $l+1/2$ does not
vanish. Hence, as long as $l+1/2$
does not vanish, the second term of \eqref{eq:commutator-expr-8}
gives a definite sign near the zero section modulo terms of lower symbolic (i.e.\ sc-decay)
order, though of the same b-decay order (hence non-compact).
\end{proof}

While one could simply (and most naturally) use a microlocalizer to a neighborhood of the
characteristic set in $\Psiscb^{0,0,0}(X)$ via using a cutoff on the second microlocal
space, $[\Tsc^*_{\pa X} X;o_{\pa X}]$, to
obtain a positive commutator, see the discussion below in the
non-real spectral parameter setting after Lemma~\ref{lemma:commutator-sc-version-imag}, one can in fact modify the
commutator (in a somewhat ad hoc manner) by
adding an additional term that gives the correct sign everywhere near the image of the
scattering zero section, i.e.\ the b-face, and we do so here.

\begin{lemma}\label{lemma:s-modified-ellipticity}
Let
$$
\tilde s=(\tilde r-1/2)\taub (\taub^2+\mub^2)^{-1}=(\tilde r-1/2)x\tau
(\tau^2+\mu^2)^{-1}.
$$

Then
$$
H_{\re\hat p(\sigma)}a+2\tilde s a \re\hat p(\sigma)
$$
is a
positive elliptic multiple of $\re\sigma$ in $S^{2\tilde r-1,2(\tilde r+l)-1,2l}$ near the image of the
scattering zero section, i.e.\ the b-face, if $l+1/2>0$, and it is a
negative elliptic multiple there if $l+1/2<0$.

In fact,
\begin{equation}\begin{aligned}\label{eq:modified-commutator-real-princ}
&H_{\re\hat p(\sigma)}a+2\tilde s a \re\hat p(\sigma)\\
&=x^{-2(l+\tilde r)+1}(\tau^2+\mu^2)^{\tilde
  r-1/2} (4(\re\sigma)(l+1/2) -2(2l+\tilde
r+1/2)\tau).
\end{aligned}\end{equation}
\end{lemma}

\begin{rem}
Due to localization near the scattering zero section added explicitly
in the discussion after the proof, the first,
sc-differentiability, order in $S^{2\tilde r-1,2(\tilde
  r+l)-1,2l}$ is actually irrelevant.

Moreover, the analogue of the conclusion remains valid with a regularizer as in
Remark~\ref{rem:regularizer-choice}, i.e.\ $a$ replaced by
$a^{(\ep)}$, provided in the definition of $\tilde s$ as well as in
the conclusion, $\tilde r$ is
replaced by \eqref{eq:tilde-r-replace} (except in the exponent), and in the conclusion an
overall factor of $f_\ep(\taub^2+\mub^2)$ is added.
\end{rem}

\begin{rem}
As the proof below shows, replacing $\tilde s$ by
$$
\hat s=2(l+\tilde r)(\taub^2+\mub^2)^{-1} \taub=2(l+\tilde r)x(\tau^2+\mu^2)^{-1} \tau
$$
replaces the right hand side of \eqref{eq:modified-commutator-real-princ} by
$$
x^{-2(l+\tilde r)+1}(\tau^2+\mu^2)^{\tilde
  r-3/2} 4(\re\sigma)\big(-(l+\tilde r) \tau^2+(l+1/2)\mu^2\big).
$$
This is manifestly definite with the same sign as $\re\sigma$ if
$l+1/2>0$, $l+\tilde r<0$, and with the opposite sign if $l+1/2<0$,
$l+\tilde r>0$, and the sign requirements for $l+\tilde r$ turn out to
be natural for the global problem, namely these give the signs
required to obtain microlocal estimates at the other radial set. However, the
terms from $\im\hat P(\sigma)$, relevant due to \eqref{eq:commutator-expr-8}, give rise to terms like $-2(2l+\tilde
r+1/2)\tau$ above in \eqref{eq:modified-commutator-real-princ} (including with $\mu_j$ in place of $\tau$), unless
stronger assumptions are imposed on $\im\hat P(\sigma)$, so in the
generality of the present paper this alternative approach is not
particularly fruitful. Nonetheless, the alternative approach becomes very useful in the
companion paper \cite{Vasy:Zero-energy-lag}, where the zero energy
limit is studied and where stronger assumptions are imposed on
$\im\hat P(\sigma)$; it is this perspective that enables us to obtain
uniform estimates as $\sigma\to 0$ in that case.
\end{rem}

\begin{proof}
Adding to \eqref{eq:commutator-sc-version}
\begin{equation*}\begin{aligned}
2\tilde sa&=2 (\tilde r-1/2) x\tau(\tau^2+\mu^2)^{-1} a\\
&=2(\tilde r-1/2)\taub (\taub^2+\mub^2)^{-1} a=2x^{-2l}(\taub^2+\mub^2)^{\tilde
  r-3/2} x^{-1}(\tilde r-1/2)\taub\\
&=2x^{-2(l+\tilde r)+1}(\tau^2+\mu^2)^{\tilde
  r-3/2} (\tilde r-1/2)\tau
\end{aligned}\end{equation*}
times $\re\hat p(\sigma)$, namely
$$
x^2\taub^2-2(\re\sigma)x\taub+x^2\mub^2=\tau^2+\mu^2-2(\re\sigma)\tau,
$$
we obtain
\begin{equation}\begin{aligned}\label{eq:modified-commutator-real-princ-bb}
&x^{-2(l+\tilde r)+1}(\tau^2+\mu^2)^{\tilde
  r-3/2} \Big(4(\re\sigma)\big((l+\tilde r) \tau^2+(l+1/2)\mu^2\big)-4(l+\tilde r)\tau(\tau^2+\mu^2)\\
&\qquad\qquad\qquad\qquad+2 (\tilde r-1/2) \tau (\tau^2+\mu^2)
-4(\re\sigma)(\tilde r-1/2) \tau^2\Big)\\
&=x^{-2(l+\tilde r)+1}(\tau^2+\mu^2)^{\tilde
  r-3/2} \Big(4(\re\sigma)(l+1/2)(\tau^2+\mu^2)\\
&\qquad\qquad\qquad\qquad \qquad\qquad\qquad\qquad-2(2l+\tilde
r+1/2)\tau(\tau^2+\mu^2)\Big)\\
&=x^{-2(l+\tilde r)+1}(\tau^2+\mu^2)^{\tilde
  r-1/2} (4(\re\sigma)(l+1/2) -2(2l+\tilde
r+1/2)\tau).
\end{aligned}\end{equation}
As already mentioned, the factor
$\tau$ is small near the scattering zero section, so $4(\re\sigma) (l+1/2) -2(2l+\tilde
r+1/2)\tau$ has the same (definite) behavior as $4(\re\sigma) (l+1/2)$ nearby, and
thus this whole expression has the same sign as
$-\re\sigma$ if $l+1/2<0$, and the same sign as $\re\sigma$ if $l+1/2>0$.
\end{proof}

Using an additional cutoff factor $\chi$ which is identically $1$ in a neighborhood
of the zero section (where the above computation already gave the
correct sign), one can combine this with standard scattering
estimates by making this factor microlocalize near the zero
section, so depending on the sign of the Hamilton derivative, there may be an error arising from the support of its
differential, but this is controlled from the incoming radial set, see
the discussion around \eqref{eq:sc-finite-pts-to-0}. (An
alternative is
instead making the factor monotone along the Hamilton flow with a
strict sign outside a small neighborhood of the radial sets.) Recall
that as the scattering decay order of this operator is $2l+2\tilde r-1$, the requirement for the incoming radial point
estimate (away from the sc-zero section), for formally self-adjoint
operators (thus ignoring the $\im\hat P(\sigma)$ terms), is $2l+2\tilde
r-1>2(-1/2)=-1$, i.e.\ $\tilde r+l>0$. This means that using such a
cutoff we have microlocal control on the support of $d\chi$ if $\tilde
r+l>0$ (and $l+1/2<0$ as above).

\begin{figure}[ht]
\begin{center}
\includegraphics[height=60mm]{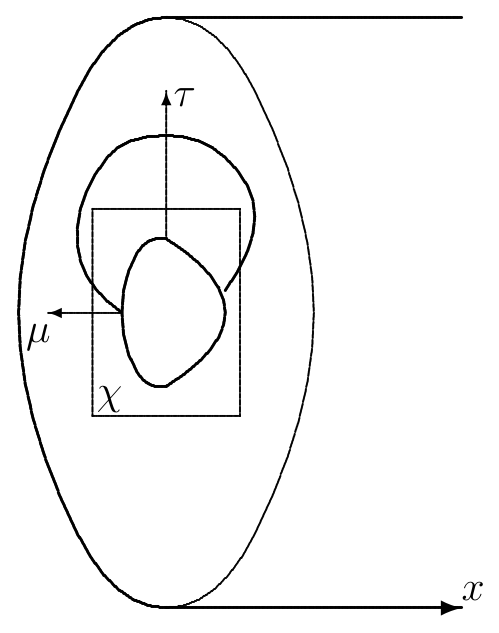}
\end{center}
\caption{The support of $\chi$ on the second microlocal space,
  indicated by the rectangular box. The characteristic set is the
  circular curve tangent to the $\mu$ axis at the b-face, given by
  the sc-zero section.}
\label{fig:sc-2-micro-loc}
\end{figure}

On the other hand, if $l+1/2>0$, then
it is not hard to see that the additional term caused by the
commutator with $\chi$ contributes a term with the same sign as the
weight term, i.e.\ a sign that agrees with that of
$\re\sigma$. Indeed, we can take
$$
\chi=\chi_0(\mu^2)\chi_1(\tau^2),
$$
with $\chi_0$, $\chi_1$ identically $1$ near $0$ of compact support
sufficiently close to $0$ and with $\chi_1$
having relatively large support so that $\supp\chi_0(.)\cap\supp d\chi_1(.)$
is disjoint from the zero set of $\re \hat p(\sigma)$. See Figure~\ref{fig:sc-2-micro-loc}. Thus, elliptic
scattering estimates control the $d\chi_1$ term. On the other hand,
doing the computation in the b-notation,
\begin{equation*}\begin{aligned}
&\{x^2(\taub^2+\mub^2)-2x\re\sigma\taub,\chi_0(x^2\mub^2)\}\\
&=2(2x^2\taub-2x\re\sigma)x^2\mub^2\chi_0'(x^2\mub^2)\\
&=-4x(\re\sigma-\tau)\mu^2\chi_0'(\mu^2),
\end{aligned}\end{equation*}
so if $\chi_1$ is arranged to have sufficiently small support, say in
$[-(\re\sigma)^2/2, (\re\sigma)^2/2]$, then this expression has the
same sign as $\re\sigma$. Arranging that $-\chi_0'$ is a square, this
simply adds another term of the correct sign to our symbolic
computation.

As non-real $\sigma$ complicates the arguments, we first consider real
$\sigma$.
We have:

\begin{lemma}\label{lemma:modified-commutator-real-princ}
Suppose $\sigma$ is real. Let $S=S^*\in\Psiscb^{-1,-1,0}(X)$ have principal
symbol $\tilde s$, $A=A^*\in\Psiscb^{-\infty,2(\tilde r+l),2l+1}(X)$
have principal symbol $\chi a$, and consider
\begin{equation}\begin{aligned}\label{eq:modified-commutator-real}
&i(\hat P(\sigma)^*A-A\hat P(\sigma))+AS\hat P(\sigma)+\hat P(\sigma)^*SA\\
&=\im \hat P(\sigma)A+A\im\hat P(\sigma)+i[\re \hat P(\sigma),A]+AS\re
\hat P(\sigma)+\re\hat P(\sigma) SA\\
&\qquad\qquad\qquad\in\Psiscb^{-\infty,2(\tilde r+l)-1,2l}(X).
\end{aligned}\end{equation}
With the notation of \eqref{eq:P-0-form} and \eqref{eq:Q-form},
the principal symbol of \eqref{eq:modified-commutator-real} is
\begin{equation}\begin{aligned}\label{eq:modified-commutator-real-princ-im-added}
x^{-2(l+\tilde r)+1}(\tau^2+\mu^2)^{\tilde
  r-1/2} \Big(4\sigma&(l+1/2-\im\alpha_+(\sigma)) \\
&-2(2l+\tilde
r+1/2-\im a_0-\sigma \im b_0)\tau\\
&+2\sum_j (\im a_j+\sigma
\im b_j)\mu_j\Big)\chi,
\end{aligned}\end{equation}
modulo terms involving derivatives of $\chi$.
\end{lemma}

\begin{rem}\label{rem:modified-commutator-real-princ}
Equation~\eqref{eq:modified-commutator-real-princ-im-added} shows that
the threshold value $-1/2$ of $l$ is shifted to
$-1/2+\im\alpha_+(\sigma)$. Note that due to the support condition on
$\chi$, the $\tau$ and $\mu_j$ terms in the parentheses can be
absorbed into $4\sigma(l+1/2-\im\alpha_+(\sigma))$ when the latter has
a definite sign, as discussed in the proofs of
Corollary~\ref{cor:commutator-sc-version} and
Lemma~\ref{lemma:s-modified-ellipticity}.

Moreover, the analogue of the conclusion remains valid with a regularizer as in
Remark~\ref{rem:regularizer-choice}, i.e.\ $a$ replaced by
$a^{(\ep)}$, and correspondingly $A$ by $A^{(\ep)}$, provided in the definition of $\tilde s$ as well as in
the conclusion, $\tilde r$ is
replaced by \eqref{eq:tilde-r-replace} (except in the exponent), and in the conclusion an
overall factor of $f_\ep(\taub^2+\mub^2)$ is added.
\end{rem}

\begin{proof}
We observe that \eqref{eq:modified-commutator-real} has principal
symbol given by \eqref{eq:modified-commutator-real-princ} times $\chi$, plus
$2\chi a$ times the principal symbol of $\im\hat P(\sigma)\in\Psiscb^{1,-1,-1}(X)$,
modulo the
term arising from the cutoff. Now, the principal symbol of $\im\hat
P(\sigma)$ is
$$
\im\Big(x^2(a_0+\sigma b_0)\taub+\sum_jx^2(a_j+\sigma
b_j)(\mub)_j-2x\sigma\alpha_+(\sigma)\Big),
$$
as follows from \eqref{eq:hat-P-form-b},
\eqref{eq:hatPQ-full-expression} and
\eqref{eq:PQ-stronger-sc}.
Thus, the principal symbol of \eqref{eq:modified-commutator-real} is
\eqref{eq:modified-commutator-real-princ-im-added} modulo terms
involving derivatives of $\chi$, as desired.
\end{proof}

The below-threshold regularity statement (so the sc-zero section is
the outgoing radial set, corresponding to low decay) is:

\begin{prop}\label{prop:outgoing-symb-est}
Suppose $l+1/2-\im\alpha_+(\sigma)<0$, $\tilde
r+l-\im\alpha_-(\sigma)>0$ and $\sigma$ is real. Then
\begin{equation}\label{eq:symbolic-est-1}
\|u\|_{\Hb^{\tilde r-1/2,l}}\leq C(\|\hat P(\sigma)u\|_{\Hb^{\tilde r-1/2,l+1}}+\|u\|_{\Hb^{-N,l}}).
\end{equation}
This estimate holds in the strong sense that if $u\in \Hb^{\tilde
  r'-1/2,l}$ for some $\tilde r'$ satisfying the inequality above in
place of $\tilde r$ and if
$\hat P(\sigma)u\in \Hb^{\tilde r-1/2,l+1}$ then
$u\in \Hb^{\tilde r-1/2,l}$ and the estimate holds.

Similarly, if $l+1/2-\im\alpha_+(\sigma)<0$, $r=\tilde
r+l-1/2>-1/2+\im\alpha_-(\sigma)$,  and $\sigma$ is real, then
\begin{equation}\label{eq:symbolic-est-scb-1}
\|u\|_{\Hscb^{s,r,l}}\leq C(\|\hat P(\sigma)u\|_{\Hscb^{s-2,r+1,l+1}}+\|u\|_{\Hscb^{-N,-N,l}}).
\end{equation}
Again this holds in the analogous sense that if $\hat P(\sigma) u$ is
in the space on the right hand side and $u\in\Hscb^{s',r',l}$ for some
$s',r'$ satisfying the inequality above with $r'$ in place of $r$, then $u$ is
a member of the space on the left hand side, and the estimate holds.
\end{prop}

\begin{proof}
Consider $\sigma>0$ for definiteness; otherwise the overall sign
switches.

At first we consider sufficiently regular $u$ so that all computations
directly make sense. Concretely, everything below works {\em directly}
if $u\in\Hb^{\tilde r,l}$, resp.\ $\Hscb^{s,r+1/2,l}$, with the loss
of $1/2$ an order in the first, resp.\ second, slot, over the
statement of the proposition arising from having to consider e.g.\
$\hat P(\sigma)^*A$ separately from the commutator. However, a very simple
regularization argument (even simpler than the one discussed below,
i.e.\ it has even less impact on the argument),
see \cite[Proof of Proposition~5.26]{Vasy:Minicourse} as well as \cite[Lemma~3.4]{Haber-Vasy:Radial}, removes this
restriction and allows $u\in\Hb^{\tilde r-1/2,l}$, resp.\
$\Hscb^{s,r,l}$ (though the regularization discussed below completely
removes the need for this, {\em just as} in the low regularity case of the aforementioned references).

The principal symbol \eqref{eq:modified-commutator-real-princ-im-added} of \eqref{eq:modified-commutator-real} can be
written as $-b^2+e$ by Remark~\ref{rem:modified-commutator-real-princ}, with $e$ arising from the Poisson bracket with
the cutoff $\chi$, thus supported on $\supp d\chi$, so we obtain
\begin{equation}\label{eq:symbolic-est-scb-8a}
i(\hat P(\sigma)^*A-A\hat P(\sigma))+AS\hat P(\sigma)+\hat P(\sigma)^*SA=-B^*B+E+F,
\end{equation}
where $B\in\Psiscb^{*,l+\tilde r-1/2,l}$ has principal symbol $b$, $E\in\Psiscb^{*,2(l+\tilde r)-1,l}$ has principal symbol $e$, and
$F\in\Psiscb^{*,2(l+\tilde r)-2,l}$ is lower order in the sc-decay
sense. Applying to $u$ and pairing with $u$ gives
\begin{equation}\label{eq:symbolic-est-scb-16}
\|Bu\|^2\leq 2|\langle P(\sigma)u,Au\rangle|+|\langle Eu,u\rangle|+|\langle Fu,u\rangle|.
\end{equation}
Here the $E$ term is controlled by
the incoming radial point and propagation
estimates (as well as the elliptic estimates, including near sc-fiber
infinity!), see \eqref{eq:sc-finite-pts-from-0}. It is helpful to
write $A=A_1^2$, $A_1=A_1^*$, as arranged by taking $\chi_0,\chi_1$ to
be squares and letting the principal symbol of $a_1$ to be the square
root of that of $A$. Then $b$ is an elliptic multiple of $x^{1/2}a_1$,
so $\|x^{1/2}A_1 u\|^2$ is controlled by $\|Bu\|^2$ modulo terms that
can be absorbed into $|\langle Fu,u\rangle|$.
Thus, modulo terms absorbed into the $F$ term,
$$
\langle
P(\sigma)u,Au\rangle=\langle x^{-1/2}A_1P(\sigma) u,x^{1/2}A_1 u\rangle
$$
is controlled by
$$
\|Bu\|\|x^{-1/2}A_1P(\sigma) u\|\leq\ep \|Bu\|^2+\ep^{-1}\|x^{-1/2}A_1P(\sigma) u\|^2,
$$
and now the first term can be absorbed into the left hand side of
\eqref{eq:symbolic-est-scb-16}. This gives, using the controlled $E$
terms, with elliptic estimates for the scattering differentiability
order, with $r=\tilde r+l-1/2$,
\begin{equation}\label{eq:symbolic-est-scb-1p}
\|u\|_{\Hscb^{s,r,l}}\leq C(\|\hat P(\sigma)u\|_{\Hscb^{s-2,r+1,l+1}}+\|u\|_{\Hscb^{-N,r-1/2,l}}).
\end{equation}
Since $\|u\|_{\Hscb^{-N,r-1/2,l}}$ can be bounded by a small multiple of $\|u\|_{\Hscb^{-N,
    r,l}}$ plus a large multiple of $\|u\|_{\Hscb^{-N,-N,l}}$, with the
former being absorbable into the left hand side, this proves
\eqref{eq:symbolic-est-scb-1}, and thus \eqref{eq:symbolic-est-1} as a
special case, under the additional assumption of membership of $u$ in the space
on the left hand side.

In fact the standard
regularization argument, using the second part of Remark~\ref{rem:modified-commutator-real-princ}, shows that the
estimate \eqref{eq:symbolic-est-scb-1p} holds in the stronger sense that if the right hand side is
finite, so is the left hand side, and iterating the estimate gives
\eqref{eq:symbolic-est-scb-1} and \eqref{eq:symbolic-est-1}.
\end{proof}

On the other hand, for $l+1/2-\im\alpha_+(\sigma)>0$ (so the sc-zero section is the
incoming radial set, corresponding to high decay) we have:

\begin{prop}\label{prop:incoming-symb-est}
Suppose $l+1/2-\im\alpha_+(\sigma)>0$, $\tilde
r+l-\im\alpha_-(\sigma)<0$ and $\sigma$ is real. Then
\begin{equation}\label{eq:symbolic-est-2}
\|u\|_{\Hb^{\tilde r-1/2,l}}\leq C(\|\hat P(\sigma)u\|_{\Hb^{\tilde r-1/2,l+1}}+\|u\|_{\Hb^{-N,l}}).
\end{equation}

Similarly, if $l+1/2-\im\alpha_+(\sigma)>0$, $r=\tilde
r+l-1/2<-1/2+\im\alpha_-(\sigma)$,  and $\sigma$ is real, then
\begin{equation}\label{eq:symbolic-est-scb-2}
\|u\|_{\Hscb^{s,r,l}}\leq C(\|\hat P(\sigma)u\|_{\Hscb^{s-2,r+1,l+1}}+\|u\|_{\Hscb^{-N,-N,l}}).
\end{equation}

These estimates hold in the sense analogous to
Proposition~\ref{prop:outgoing-symb-est}, except there is no need for
$\tilde r'$, resp.\ $r'$ to satisfy any inequalities, since for
sufficiently negative $\tilde r',r'$, as one may always assume in this
context, the inequalities involving these are
automatically satisfied.
\end{prop}

\begin{rem}
Note that in this proposition there is no limit to background decay,
as represented by $\tilde r',r'$,
thus technically to regularizability,
unlike what happens in standard radial point estimates, see
\cite{RBMSpec,Vasy-Dyatlov:Microlocal-Kerr,Haber-Vasy:Radial,Vasy:Minicourse}. The
reason is that the analogue of the limitation of regularizability in
the standard setting is the b-decay order, which we here fix, thus we
do {\em not} improve it over a priori expectations.
\end{rem}

\begin{proof}
In this case the cutoff term $H_{\re\hat p(\sigma)}\chi$ is also
principally positive (for $\sigma>0$, otherwise there is an overall
sign switch, though that has no impact on the argument), so
the principal symbol of \eqref{eq:modified-commutator-real} is
$b^2+b_1^2$, and we obtain that \eqref{eq:symbolic-est-scb-8a} is
replaced by
$$
i(\hat P(\sigma)^*A-A\hat P(\sigma))+AS\hat P(\sigma)+\hat P(\sigma)^*SA=B^*B+B_1^*B_1+F,
$$
where $B,B_1\in\Psiscb^{*,l+\tilde r-1/2,l}$ have principal symbol $b,b_1$, and
$F\in\Psiscb^{*,2(l+\tilde r)-2,l}$ is lower order in the sc-decay
sense.
Combining this with the outgoing radial point and propagation
estimates (as well as the elliptic estimates) as in
\eqref{eq:sc-finite-pts-to-0}, we can proceed as in the proof of
Proposition~\ref{prop:outgoing-symb-est} to conclude \eqref{eq:symbolic-est-2}
as well as \eqref{eq:symbolic-est-scb-2}.
\end{proof}

Now, the last term of \eqref{eq:symbolic-est-1} and of \eqref{eq:symbolic-est-2} can be estimated using the normal operator
$$
N(\hat P(\sigma)) =-2\sigma \Big(x^2D_x+i\frac{n-1}{2}x
+\alpha_+(\sigma) x\Big),
$$
noting that
$$
x^{-1}N(\hat P(\sigma)) =-2\sigma \Big(xD_x+i\frac{n-1}{2}
+\alpha_+(\sigma) \Big)
$$
should be considered as an operator from $\Hb^{-N,l}$ to
$\Hb^{-N-1,l}$.

Concretely we have:

\begin{lemma}\label{lemma:simple-normal}
For $l<-1/2+\im\alpha_+(\sigma)$, we have
$$
\|v\|_{\Hb^{-N,l}}\leq C\|N(\hat P(\sigma)) v\|_{\Hb^{-N,l+1}}
$$
whenever $v\in \Hb^{-N,l}$ and $N(\hat P(\sigma)) v\in \Hb^{-N,l+1}$.

The same estimate also holds for $l>-1/2+\im\alpha_+(\sigma)$.
\end{lemma}

\begin{rem}
Here $\im\alpha_+(\sigma)$ is a function on $\pa X$, and the
inequalities $l<-1/2+\im\alpha_+(\sigma)$, resp.\
$l>-1/2+\im\alpha_+(\sigma)$, need to hold at each point of $\pa X$.
\end{rem}

\begin{rem}\label{rem:decay-improve}
It is straightforward to formalize and prove, via a contour shifting
argument on the Mellin transform side, a version of this lemma that
assumes that $v$ is supported in $x\leq 1$, say (as relevant below in
the setting of Proposition~\ref{prop:real-sigma-Fredholm}), and that
$v\in\Hb^{-N,l'}$ only for some $l'<l$, with $l'$ satisfying the same
inequality as $l$, and concludes that
$v\in\Hb^{-N,l}$. However, we do not need this in the present paper,
and in any case one can run such an argument as an a posteriori
`regularity' (here meaning b-decay) argument.
\end{rem}

\begin{proof}
For our immediate purposes
it is more convenient to work with $L^2_{\bl}$, so we set $\Hbb$
to be the b-Sobolev space relative to $L^2_{\bl}$, here this really is
of interest in $[0,\infty)\times\pa X$, with density
$\frac{dx}{x}\,dg_{\pa X}$.
Since the quadratic form on $L^2_\bl$ is $\langle
x^{n}\cdot,\cdot\rangle_{g_0}$, $L^2_\bl=x^{-n/2}L^2$, so $x^{-1}N(\hat
P(\sigma))$ mapping from $\Hb^{-N,l}$ to $\Hb^{-N-1,l}$ amounts to
$x^{-n/2}x^{-1}N(\hat P(\sigma))x^{n/2}$ being considered from
$\Hbb^{-N,l}$ to $\Hbb^{-N-1,l}$, or $x^{-n/2-l}x^{-1}N(\hat
P(\sigma))x^{n/2+l}$ from
$\Hbb^{-N,0}$ to $\Hbb^{-N-1,0}$. But this is
$$
-2\sigma \Big(xD_x-i(l+1/2)
+\alpha_+(\sigma) \Big),
$$
which on the Mellin transform side is multiplication by
\begin{equation}\label{eq:MT-mult-op}
-2\sigma \Big(\taub-i(l+1/2)
+\alpha_+(\sigma)\Big),
\end{equation}
which is invertible for real $\alpha_+$ if $l+1/2\neq 0$, and in general
if $l+1/2\neq\im\alpha_+(\sigma)$. The differential order is not
an issue: the Mellin transform, with image restricted to the real line, is an isomorphism from the Sobolev spaces $\Hbb^{s',0}$ on $[0,\infty)\times\pa X$
and the large parameter, in $\taub$, i.e.\ semiclassical in the
reciprocal $\langle\taub\rangle^{-1}$, Sobolev spaces
$H^{s'}_{\langle\taub\rangle^{-1}}$, see \cite{Melrose:Atiyah} around
equation (5.41) and \cite{Vasy-Dyatlov:Microlocal-Kerr} around equation
(3.8). Since the multiplication operator \eqref{eq:MT-mult-op}, which
is multiplication by a constant for each fixed $\taub$, has a bounded inverse on these
spaces when $l+1/2\neq \im\alpha_+(\sigma)$, the conclusion follows.
\end{proof}

\begin{prop}\label{prop:real-sigma-Fredholm}
Suppose $l+1/2-\im\alpha_+(\sigma)<0$, $\tilde
r+l-\im\alpha_-(\sigma)>0$ and $\sigma\neq 0$ is real. Then
$$
\|u\|_{\Hb^{\tilde r-1/2,l}}\leq C(\|\hat P(\sigma)u\|_{\Hb^{\tilde r-1/2,l+1}}+\|u\|_{\Hb^{-N,l-\delta}}).
$$
This estimate holds in the sense that if $u\in \Hb^{\tilde
  r'-1/2,l}$ for some $\tilde r'$ satisfying the inequality above in
place of $\tilde r$ and if
$\hat P(\sigma)u\in \Hb^{\tilde r-1/2,l+1}$ then
$u\in \Hb^{\tilde r-1/2,l}$ and the estimate holds.

Similarly, if $l+1/2-\im\alpha_+(\sigma)<0$, $r=\tilde
r+l-1/2>-1/2+\im\alpha_-(\sigma)$,  and $\sigma$ is real, then
$$
\|u\|_{\Hscb^{s,r,l}}\leq C(\|\hat P(\sigma)u\|_{\Hscb^{s-2,r+1,l+1}}+\|u\|_{\Hscb^{-N,-N,l-\delta}}).
$$
Again this holds in the analogous sense that if $\hat P(\sigma) u$ is
in the space on the right hand side and $u\in\Hscb^{s',r',l}$ for some
$s',r'$ satisfying the inequality above with $r'$ in place of $r$, then $u$ is
a member of the space on the left hand side, and the estimate holds.

The analogous conclusions also hold if $l+1/2-\im\alpha_+(\sigma)>0$, $\tilde
r+l-\im\alpha_-(\sigma)<0$, $r=\tilde
r+l-1/2<-1/2+\im\alpha_-(\sigma)$, except that $\tilde r',r'$ do not
need to satisfy any inequalities, cf.\ Proposition~\ref{prop:incoming-symb-est}.
\end{prop}

\begin{proof}
Applying Lemma~\ref{lemma:simple-normal} with $v=\psi u$, $\psi$ supported near $x=0$,
identically $1$ in a smaller neighborhood, using $\hat
P(\sigma)-N(\hat P(\sigma))\in\Psib^{2,-1-\delta}(X)$ (note the $\delta>0$ extra
order of decay!),
$$
\|u\|_{\Hb^{-N,l}}\leq C(\|\hat P(\sigma) u\|_{\Hb^{-N,l+1}}+\|u\|_{\Hb^{-N+2,l-\delta}}).
$$
In combination with \eqref{eq:symbolic-est-1} we have
$$
\|u\|_{\Hb^{\tilde r-1/2,l}}\leq C(\|\hat P(\sigma)u\|_{\Hb^{\tilde r-1/2,l+1}}+\|u\|_{\Hb^{-N+2,l-\delta}}),
$$
where $-N+2$ may simply be replaced by $-N$ in the notation, giving
the first statement of the proposition.

Using the second microlocal estimate \eqref{eq:symbolic-est-scb-1}
instead of \eqref{eq:symbolic-est-1} gives, completely analogously,
the second statement of the proposition.

The reversed inequality version on the orders is completely analogous.
\end{proof}

\begin{proof}[Proof of Theorem~\ref{thm:main} for real $\sigma$.]
We start by showing a slight improvement of the statement of Proposition~\ref{prop:real-sigma-Fredholm}.
Namely, as soon
as the nullspace of $P(\sigma)$ is trivial, the usual argument allows the last
relatively compact term in Proposition~\ref{prop:real-sigma-Fredholm} to be dropped, so that
$$
\|u\|_{\Hb^{\tilde r-1/2,l}}\leq C\|\hat P(\sigma)u\|_{\Hb^{\tilde r-1/2,l+1}},
$$
and this is uniform for $\sigma$ in compact sets in
$\RR\setminus\{0\}$. Indeed, if this is not true, there are sequences
$\sigma_j$ in the fixed compact set,
$u_j\in \Hb^{\tilde r-1/2,l}$, which we may normalize to
$\|u_j\|_{\Hb^{\tilde r-1/2,l}}=1$, with $\hat P(\sigma_j)u_j\in
\Hb^{\tilde r-1/2,l+1}$ such that
$$
1=\|u_j\|_{\Hb^{\tilde r-1/2,l}}>j\|\hat P(\sigma_j)u_j\|_{\Hb^{\tilde r-1/2,l+1}},
$$
so $\hat P(\sigma_j)u_j\to 0$ in $\Hb^{\tilde r-1/2,l+1}$. But by the
weak compactness of the unit ball, there is a weakly convergent
subsequence, which we do not indicate in notation, converging to some $u\in \Hb^{\tilde r-1/2,l}$ and one may also
assume that $\sigma_j$ also converges (by passing to another
subsequence). In particular, due to the compactness of the inclusion
$\Hb^{\tilde r-1/2,l}\to \Hb^{-N,l-\delta}$, $u_j$ converges to $u$
in $\Hb^{-N,l-\delta}$ strongly, so by the first estimate of
Proposition~\ref{prop:real-sigma-Fredholm}, using that the first term
on the right hand side goes to $0$, we conclude that
$$
1\leq C\|u\|_{\Hb^{-N,l-\delta}},
$$
so in particular $u\neq 0$. On the other hand, $\hat P(\sigma_j)u_j\to
\hat P(\sigma) u$ in $\Hb^{\tilde r-5/2-\ep,l+1-\ep}$, $\ep>0$, so we
conclude that $\hat P(\sigma) u=0$, which contradicts the triviality
of the nullspace.

If $P(\sigma)=P(\sigma)^*$, the triviality of the nullspace, on the other
hand, follows from the standard results involving the absence of
embedded eigenvalues: the results thus far, as $\tilde r+l>0$ is
arbitrary, show that any element of the nullspace is in fact in
$\Hb^{\infty,l}$, i.e.\ is conormal. Then a generalized and extended version of the
boundary pairing formula of \cite{RBMSpec}, using the approach of
Isozaki \cite{IsoRad}, as given in
\cite[Proposition~7]{Vasy:Self-adjoint} (the Feynman and anti-Feynman function
spaces correspond to the incoming and outgoing resolvents), shows that in fact it is in $\dCI(X)$ and then
unique continuation arguments at infinity conclude the proof. Note
that if $P(\sigma)\neq P(\sigma)^*$, the uniformity of our estimates
still implies that for $P(\sigma)-P(\sigma)^*$ suitably small, the
triviality of the nullspace holds.

Notice that for $\hat P(\sigma)^*$ has the same properties as $\hat
P(\sigma)$ except that we need to replace
$\im\alpha_\pm(\sigma)$ by their negatives, so
actually we have proved two estimates
\begin{equation}\begin{aligned}\label{eq:direct-op-est}
\|u\|_{\Hb^{\tilde r-1/2,l}}\leq &C\|\hat P(\sigma)u\|_{\Hb^{\tilde
    r-1/2,l+1}},\\
&\qquad u\in \Hb^{\tilde r-1/2,l},\ \hat P(\sigma)u\in \Hb^{\tilde
    r-1/2,l+1},
\end{aligned}\end{equation}
\begin{equation}\begin{aligned}\label{eq:adjoint-op-est}
\|v\|_{\Hb^{\tilde r'-1/2,l'}}\leq &C\|\hat P(\sigma)^*v\|_{\Hb^{\tilde
    r'-1/2,l'+1}},\\
&\qquad v\in \Hb^{\tilde r'-1/2,l'},\ \hat P(\sigma)^*v\in \Hb^{\tilde
    r'-1/2,l'+1},
\end{aligned}\end{equation}
where we may take
$$
l<-1/2+\im\alpha_+(\sigma),\ \tilde r+l-\im\alpha_-(\sigma)>0,\
\tilde r'-1/2=-(\tilde r-1/2),\ l'=-l-1,
$$
for then
\begin{equation*}\begin{aligned}
\tilde r'+l'+\im\alpha_-(\sigma)&=-\tilde r+1-l-1+\im\alpha_-(\sigma)<0,\\
l'&=-1/2-(l+1/2)>-1/2-\im\alpha_+(\sigma),
\end{aligned}\end{equation*}
and now the spaces on the left, resp.\ right,
hand side of \eqref{eq:direct-op-est} and right, resp.\ left, hand
side of \eqref{eq:adjoint-op-est}
are duals of each other.

There is a slight subtlety
in that we only have
\begin{equation}\label{eq:P-sigma-Hb-mapping}
\hat P(\sigma):\cX_\sigma=\{u\in \Hb^{\tilde r-1/2,l}:\ \hat P(\sigma)
u\in \Hb^{\tilde r-1/2,l+1}\}\to\cY=\Hb^{\tilde r-1/2,l+1},
\end{equation}
rather than $\Hb^{\tilde r-1/2,l}\to \Hb^{\tilde r-1/2,l+1}$, but the
treatment of this is standard, as in
\cite[Section~2.6]{Vasy-Dyatlov:Microlocal-Kerr} and \cite[Section~4.3]{Vasy:Microlocal-AH}. Indeed, certainly
injectivity is immediate from \eqref{eq:direct-op-est}. For
surjectivity note that \eqref{eq:adjoint-op-est} implies that given
$f$ in the dual of $\Hb^{\tilde r'-1/2,l'}$, which is $\Hb^{\tilde
  r-1/2,l+1}$, there exists $u$ in the dual of $\Hb^{\tilde
  r'-1/2,l'+1}$, which is $\Hb^{\tilde r-1/2,l}$ such that $\hat
P(\sigma)u=f$. To see this claim, one considers the conjugate linear functional
$v\mapsto\langle f,v\rangle$, defined for $v\in\Hb^{\tilde r'+3/2,l'}$
(so $\hat P(\sigma)\in\Hb^{\tilde r'-1/2,l'+1}$ automatically)
which by
\eqref{eq:adjoint-op-est} satisfies $|\langle f,v\rangle|\leq C\|\hat
P(\sigma)^*v\|_{\Hb^{\tilde r'-1/2,l'+1}}$; thus we can consider the
conjugate linear functional from the range
of $\hat
P(\sigma)^*$ on $\Hb^{\tilde r'+3/2,l'}$ to $\Cx$ given by $\hat
P(\sigma)^*v\mapsto \langle f,v\rangle$ which is therefore continuous when
the range is equipped with the $\Hb^{\tilde r'-1/2,l'+1}$ norm. By
the Hahn-Banach theorem, it can be extended to $\Hb^{\tilde
  r'-1/2,l'+1}$, i.e.\ there exists an element $u$ of the dual space
$\Hb^{\tilde r-1/2,l}$ such that $\langle u, \hat
P(\sigma)^*v\rangle=\langle f,v\rangle$ for all $v\in \Hb^{\tilde
  r'+3/2,l'}$, in particular for all Schwartz $v$, which is to say
$\hat P(\sigma)u=f$. But then $\hat
P(\sigma)u \in \Hb^{\tilde
  r-1/2,l+1}$, so $u\in\cX_\sigma$, showing surjectivity.
This establishes the
invertibility of $\hat P(\sigma)$ as stated.

The case of second microlocal spaces is completely analogous, and
gives the invertibility of $\hat P(\sigma)$ as a map
\begin{equation}\label{eq:P-sigma-Hscb-mapping}
\hat P(\sigma):\cX_\sigma=\{u\in \Hscb^{s,r,l}:\ \hat P(\sigma)
u\in \Hscb^{s-2,r+1,l+1}\}\to\cY=\Hscb^{s-2,r+1,l+1}.
\end{equation}
\end{proof}

\begin{rem}
We remark here that $\cX_\sigma$ given by
\eqref{eq:P-sigma-Hscb-mapping} is easily seen to have the property
that $\Hscb^{s,r+1,l}$ (which is a subspace of it) is dense in it.
Indeed, one simply needs to show regularizability in the second,
sc-decay, order. This is accomplished by taking a family
$\Lambda_\ep\in\Psiscb^{0,-1,0}(X)$ uniformly bounded in
$\Psiscb^{0,0,0}(X)$, converging to $\Id$ in $\Psiscb^{0,\ep,0}(X)$,
$\ep>0$. Now, for $u\in\cX_\sigma$ we have
$\Lambda_\ep u\to u$ in $\Hscb^{s,r,l}$ (which follows from
$\Lambda_\ep\to \Id$ strongly on $\Hscb^{s,r,l}$), and similarly
$\Lambda_\ep\hat P(\sigma)u\to\hat P(\sigma)u$ in
$\Hscb^{s-2,r+1,l+1}$. However, regarding $\hat P(\sigma)u$, what we
must actually show is that $\hat P(\sigma)\Lambda_\ep u\to\hat P(\sigma)u$ in
$\Hscb^{s-2,r+1,l+1}$. But $\hat
P(\sigma)\Lambda_\ep u=\Lambda_\ep\hat P(\sigma)u+[\hat
P(\sigma),\Lambda_\ep]u$, with $[\hat
P(\sigma),\Lambda_\ep]$ uniformly bounded in a space with one
additional order of sc-decay (and differential order!) relative to the products, namely
$\Psiscb^{1,-1,-1}(X)$, converging to $0$ in
$\Psiscb^{1,-1+\ep,-1}(X)$, so $[\hat
P(\sigma),\Lambda_\ep]u\to 0$ in $\Hscb^{s-1,r+1,l+1}$, and thus in
$\Hscb^{s-2,r+1,l+1}$. This shows $\Lambda_\ep u \to u$ in
$\cX_\sigma$, so $\Hscb^{s,r+1,l}$ is dense in $\cX_\sigma$. Since the
inclusion map $\Hscb^{s,r+1,l}\to\cX_\sigma$ is continuous, and since
$\CI(X)$ is dense in $\Hscb^{s,r+1,l}$, we conclude that $\dCI(X)$ is
also dense in it.

As for $\cX_\sigma$ in \eqref{eq:P-sigma-Hb-mapping}, one can show the
density statement by noting that if $u\in\Hb^{\tilde r-1/2,l}$ with
$\hat P(\sigma)u\in \Hb^{\tilde r-1/2,l+1}$ then $u\in\Hscb^{\tilde
  r-1/2,\tilde r+l-1/2,l}$ with $\hat P(\sigma)u\in \Hscb^{\tilde
  r-1/2,\tilde r+l+1/2,l+1}$. This is almost a special case of the
above discussion taking $s=\tilde r+3/2$, $r=\tilde r+l-1/2$, with the
only issue being that $u\in\Hscb^{s-2,r,l}$ (and $\hat
P(\sigma)u\in\Hscb^{s-2,r+1,l+1}$) not $u\in\Hscb^{s,r,l}$. But this is
easily overcome: $\hat
P(\sigma)u\in\Hscb^{s-2,r+1,l+1}$ and ellipticity of $\hat P(\sigma)$ in
the first order shows that $u\in\Hscb^{s,r,l}$. Thus, the argument of
the previous paragraph is applicable, and shows that $\dCI(X)$ is
dense in $\cX_\sigma$. It also shows that even though elements of
$\cX_\sigma$ only have a priori differential regularity $\tilde
r-1/2$, in fact, in the scattering sense, they have differential regularity
$\tilde r+3/2$.
\end{rem}

We now turn to the case of not necessarily real $\sigma$. {\em We
  remark that the regularization issues and the ways of dealing with them are completely analogous to
  the real $\sigma$ case, and we will not comment on these explicitly.}
As we have already seen, near the zero section the term $-2\sigma\tau$
is the most important part of the principal symbol since the other
terms vanish quadratically at the zero section, so it is useful to
consider
$$
\tilde P(\sigma)=\sigma^{-1}\hat P(\sigma),
$$
so
$$
\tilde p(\sigma)=\sigma^{-1}\hat p(\sigma)=-2\tau+\overline{\sigma}|\sigma|^{-2}(\tau^2+\mu^2),
$$
hence
$$
\re\tilde p(\sigma)=-2\tau+(\re\sigma)|\sigma|^{-2}(\tau^2+\mu^2),
$$
$$
\im\tilde p(\sigma)=-(\im\sigma)|\sigma|^{-2}(\tau^2+\mu^2).
$$
Thus, $\im\tilde p(\sigma)\leq 0$ if $\im\sigma\geq 0$, which means
one can propagate estimates forwards along the Hamilton flow of
$\re\tilde p(\sigma)$; similarly, if $\im\sigma\leq 0$, one can
propagate estimates backwards along the Hamilton flow of $\re\tilde
p(\sigma)$. As we have seen, for $\im\sigma\neq 0$, the operator is
actually only characteristic at the front face.
The principal symbol computation replacing Lemma~\ref{lemma:commutator-sc-version} is:

\begin{lemma}\label{lemma:commutator-sc-version-imag}
Let $A_0\in\Psiscb^{2\tilde r-1,2(\tilde r+l),2l+1}$ have principal symbol $a$ given by \eqref{eq:symb-A0-def}.
The principal symbol $H_{\re\tilde p(\sigma)}a$ of $i[\re \tilde P(\sigma),A_0]$ in $\Psiscb^{2\tilde
  r,2(\tilde r+l)-1,2l}(X)$ is
\begin{equation}\begin{aligned}\label{eq:commutator-sc-version-imag}
&x^{-2l}(\taub^2+\mub^2)^{\tilde
  r-3/2}\Big(4\big((l+\tilde r)
\taub^2+(l+1/2)\mub^2\big)-4\frac{\re\sigma}{|\sigma|^2} x(l+\tilde
r)\taub(\taub^2+\mub^2)\Big)\\
&=x^{-2(l+\tilde r)+1}(\tau^2+\mu^2)^{\tilde
  r-3/2}\Big(4\big((l+\tilde r) \tau^2+(l+1/2)\mu^2\big)-4\frac{\re\sigma}{|\sigma|^2} (l+\tilde r)\tau(\tau^2+\mu^2)\Big).
\end{aligned}\end{equation}
\end{lemma}

\begin{proof}
We have
\begin{equation*}\begin{aligned}
&\Big\{\frac{\re\sigma}{|\sigma|^2} x^2(\taub^2+\mub^2)-2x\taub,x^{-2l-1}(\taub^2+\mub^2)^{\tilde
  r-1/2}\Big\}\\
&\qquad=(2\frac{\re\sigma}{|\sigma|^2} x^2\taub-2x)(-2l-1) x^{-2l-1} (\taub^2+\mub^2)^{\tilde
  r-1/2}\\
&\qquad\qquad-(2\frac{\re\sigma}{|\sigma|^2} x^2(\taub^2+\mub^2) -2x\taub) x^{-2l-1}2(\tilde r-1/2)\taub (\taub^2+\mub^2)^{\tilde
  r-3/2}.
\end{aligned}\end{equation*}
Expanding and rearranging,
\begin{equation*}\begin{aligned}
&=4(l+1/2) x^{-2l}(\taub^2+\mub^2)^{\tilde
  r-1/2}\\
&\qquad+4(\tilde r-1/2)x^{-2l}\taub^2 (\taub^2+\mub^2)^{\tilde
  r-3/2}\\
&\qquad-4\frac{\re\sigma}{|\sigma|^2} (l+1/2)x^{-2l+1}\taub (\taub^2+\mub^2)^{\tilde
  r-1/2}\\
&\qquad-4\frac{\re\sigma}{|\sigma|^2} (\tilde r-1/2)x^{-2l+1}\taub (\taub^2+\mub^2)^{\tilde
  r-1/2}\\
&=x^{-2l}(\taub^2+\mub^2)^{\tilde
  r-3/2}\Big(4\big((l+1/2) (\taub^2+\mub^2)+(\tilde
r-1/2)\taub^2\big)\\
&\qquad\qquad \qquad\qquad\qquad\qquad \qquad\qquad-4\frac{\re\sigma}{|\sigma|^2} x(l+\tilde r)\taub(\taub^2+\mub^2)\Big)\\
&=x^{-2l}(\taub^2+\mub^2)^{\tilde
  r-3/2}\Big(4\big((l+\tilde r) \taub^2+(l+1/2)\mub^2\big)-4\frac{\re\sigma}{|\sigma|^2} x(l+\tilde r)\taub(\taub^2+\mub^2)\Big).
\end{aligned}\end{equation*}
Rewriting from the second microlocal perspective, substituting $\tau=x\taub$,
$\mu=x\mub$, completes the proof.
\end{proof}

For a moment, let us ignore the contributions to $\im\tilde P(\sigma)$
from subprincipal terms.
Again, the $(l+1/2)\mu^2$ term is the dominant one in the expression
on the right hand side of \eqref{eq:commutator-sc-version-imag}, so the commutator
has a sign that agrees with that of $l+1/2$. Since the imaginary part
has the same (indefinite) sign as $-\im\sigma$, this means that for
$\im\sigma>0$ when this commutator is negative, i.e.\ $l+1/2<0$, the
two signs agree, and one has the desired estimate; a similar
conclusion holds if $\im\sigma<0$ and $l+1/2>0$.
We can ensure the negativity/positivity of the parenthetical term (in
terms of a multiple of $\tau^2+\mu^2$ which, or whose negative, is bounded below by a
positive constant) by using a cutoff on the blown up space, in
$\frac{\tau^2}{\tau^2+\mu^2}$ (or $\tau^2/\mu^2$), which is
identically $1$ near $0$ and has small support and
whose differential is supported in the elliptic set: $\chi_3(\frac{\tau^2}{\tau^2+\mu^2})$.
We do need to add a
cutoff to localize near the scattering zero section, but as the
operator is elliptic outside the zero section, on the differential of
such a cutoff we have elliptic estimates, so these terms are
controlled. See Figure~\ref{fig:sc-2-micro-cutoff}.

In more detail, the full computation then involves
\begin{equation}\begin{aligned}\label{eq:modified-commutator-imag}
i(\tilde P(\sigma)^*A-A\tilde P(\sigma))=(\im \tilde P(\sigma) A+A\im \tilde P(\sigma))+i[\re \tilde P(\sigma),A]
\end{aligned}\end{equation}
where $A\in\Psiscb^{2\tilde r-1,2(\tilde r+l),2l+1}(X)$ as before,
namely has principal symbol $\chi a$ with
$$
\chi=\chi_0(\mu^2)\chi_1(\tau^2) \chi_3\Big(\frac{\tau^2}{\tau^2+\mu^2}\Big).
$$
Now, the first term of
\eqref{eq:modified-commutator-imag} has the correct sign at the
principal symbol level as already discussed (when $l<-1/2$ and
$\im\sigma>0$, as well as when $l>-1/2$ and $\im\sigma<0$, and when
the subprincipal terms of \eqref{eq:modified-commutator-imag} are ignored).
However, as it has one order less
sc-decay than the main term (but it degenerates as $\im\sigma\to 0$), it
is useful to write it somewhat differently.

\begin{figure}[ht]
\begin{center}
\includegraphics[height=60mm]{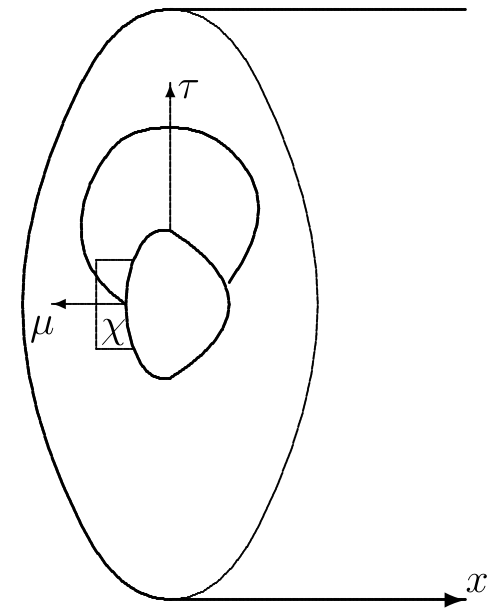}
\end{center}
\caption{The support of $\chi$ on the second microlocal space,
  indicated by the rectangular box. The characteristic set is the
  circular curve tangent to the $\mu$ axis at the b-face, given by
  the sc-zero section.}
\label{fig:sc-2-micro-cutoff}
\end{figure}

\begin{lemma}\label{lemma:im-tilde-P-positive}
We have
$$
\im\tilde P(\sigma)=-(\im\sigma)T(\sigma)+W(\sigma)
$$
with
$$
T(\sigma)=T=\sum_j
T_j^2+\sum_j T_j T'_j+\sum_j T'_j T_j+T''
$$
with $T_j=T_j^*\in\Psiscb^{1,0,-1}(X)$ (where $T_j$ is $|\sigma|^{-1}$
times the $T_j$ of \eqref{eq:hat-P-0-nonnegative}),
$T'_j=(T'_j)^*\in\Psiscb^{0,-1,-1}(X)$,
$T''_j=(T''_j)^*\in\Psiscb^{0,-2,-2}(X)$, $W=W^*\in
\Psiscb^{0,-1,-1}(X)$, so $T'_j,W$ are one order lower
than $T$ in terms of sc-decay, $T''$ two orders lower, and where,
with the notation of \eqref{eq:P-0-form} and \eqref{eq:Q-form},
$W(\sigma)$ has principal symbol
$$
\im\Big(x^2(\sigma^{-1}a_0+b_0)\taub+\sum_jx^2(\sigma^{-1}a_j+
b_j)(\mub)_j-2x\alpha_+(\sigma)\Big).
$$
\end{lemma}

\begin{proof}
This is an immediate consequence of \eqref{eq:hat-P-form},
\eqref{eq:hat-P-0-nonnegative} and \eqref{eq:hat-Q-module-form}.
\end{proof}

We now prove

\begin{prop}\label{prop:imag-control}
For $l<-1/2+\im(\alpha_+(\sigma))$ and $\im\sigma>0$, as well as for
$l>-1/2+\im(\alpha_+(\sigma))$ and $\im\sigma<0$, with $\tilde r,r$
arbitrary in either case, we have the estimates
$$
\|u\|_{\Hb^{\tilde r-1/2,l}}\leq C(\|\hat P(\sigma)u\|_{\Hb^{\tilde r-1/2,l+1}}+\|u\|_{\Hb^{-N,l-\delta}})
$$
and
\begin{equation}\label{eq:imag-control-sc-b}
\|u\|_{\Hscb^{s,r,l}}\leq C(\|\hat P(\sigma)u\|_{\Hscb^{s-2,r+1,l+1}}+\|u\|_{\Hb^{-N,l-\delta}}).
\end{equation}
These estimates hold in the sense that if $u\in \Hb^{\tilde r'-1/2,l}$,
resp.\ $u\in\Hscb^{s',r',l}$ for some $s',\tilde r',r'$, and if $\hat
P(\sigma)u$ is in the space indicated on the right hand side, then $u$
is in the space indicated on the left hand side and the estimate holds.
\end{prop}

\begin{rem}\label{rem:imag-control}
The proof below directly strengthens the norm on the left hand side of
\eqref{eq:imag-control-sc-b} to
$\|u\|_{\Hscb^{s,r,l}}+\|u\|_{\Hscb^{s,r+1/2,l-1/2}}$ thanks to the
$\|T_jA_1 u\|^2$ terms in \eqref{eq:im-sigma-expand}: at each point on
the lift of $\Tsc^*_{\pa X}X$ to the second microlocal space
$[\Tsc^*X;o_{\pa X}]$, one of
the $T_jA_1$ is an elliptic element of $\Psiscb^{*,\tilde r+l,l-1/2}$. This estimate could be
further strengthened by estimating $\tilde P(\sigma)u$ on the left hand side of
\eqref{eq:im-sigma-expand} in a corresponding dual space.

In fact, perhaps the most systematic way of approaching this problem
is to blow up the corner of $[\Tsc^*X;o_{\pa X}]$ at the intersection of
the front face (the b-face) and the lift of $\Tsc^*_{\pa X}X$, or
equivalently, and in an analytically better manner for the same
reasons as discussed regarding second microlocalization in Section~\ref{sec:pseudo}, the corner of $[\overline{\Tb^*}X;\pa_2 \overline{\Tb^*}X]$
at the intersection of the lift of $\overline{\Tb^*}_{\pa X}X$ and the
front face (the sc-face). The symbol, pseudodifferential and Sobolev spaces will have four orders, with a new
order arising from the symbolic orders at the new front face. A simple
computation shows that the vector fields $x^{1/2} xD_x,
x^{1/2}D_{y_j}$ are altogether elliptic in the interior of the new
front face, and thus this new front face corresponds to the scattering
algebra in $x^{1/2}$. Notice, however, that the above order convention makes $x^{1/2}$
order $-1$ at the new face: it vanishes to order $1/2$ at each of the two boundary
hypersurfaces whose intersection is blown up. Thus the symbolic calculus gains the defining
function of the new face at each step, so for instance principal
symbols are so defined; this corresponds to a gain of $x^{1/2}$ in the
$x^{1/2}$-scattering algebra. Then $\tilde P(\sigma)$ is in the set of
pseudodifferential operators of order (without giving a name to the
space) $2,0,-1,-1$.  Then
$\|u\|_{\Hscb^{s,r,l}}+\|u\|_{\Hscb^{s,r+1/2,l-1/2}}$ is equivalent to
the norm with orders $s,r+1/2,r+l,l$ (the order at the new face is the
sum of the sc-b orders at the two adjacent faces, and for both terms these
are the same, so microlocally there the two terms are equivalent), and the result is
an estimate (without giving a name to the space)
$$
\|u\|_{s,r+1/2,r+l,l}\leq C(\|\tilde P(\sigma)u\|_{s-2,r+1/2,r+l,l+1}+\|u\|_{\Hb^{-N,l}}).
$$
Note that for an elliptic operator (in every sense) of the same order
as $\tilde P(\sigma)$ the norm of the
first term on the
right hand side would be of type $s-2,r+1/2,r+l+1,l+1$, so the only sense in
which the estimate is not an elliptic estimate is at the new front
face, where there is a loss of an order corresponding to real
principal type estimates; indeed, $T(\sigma)$ is a subprincipal term there, so the
two terms of \eqref{eq:modified-commutator-imag} have the same order
at the front face. A useful feature then is that the characteristic
set at the new face is purely described by $\re \tilde P(\sigma)$, and
is independent of $\sigma$, so in this formulation the Fredholm theory
is on fixed spaces for all $\sigma$ with $\im\sigma>0$.

However,
these estimates do not extend uniformly to real $\sigma$, which is our
goal in the subsequent proof of Theorem~\ref{thm:main}, so we do not
develop this theory further here.
\end{rem}

\begin{proof}
We start by considering the term
\begin{equation}\label{eq:imag-control-6}
\im \tilde P(\sigma) A+A\im \tilde P(\sigma)
\end{equation}
of \eqref{eq:modified-commutator-imag} and use
Lemma~\ref{lemma:im-tilde-P-positive}, first dealing with the
$T(\sigma)$ term, namely
\begin{equation}\label{eq:imag-control-6a}
-(\im\sigma)(T(\sigma)A+AT(\sigma)).
\end{equation}
As before, we will apply this to $u$ and take the inner product with
$u$, resulting in
$$
-(\im\sigma)\langle (T(\sigma)A+AT(\sigma))u,u\rangle.
$$

We again write
$A=A_1^2$, $A_1=A_1^*$, $A_1\in \Psiscb^{\tilde r-1/2,\tilde
  r+l,l+1/2}$ which one can certainly do by choosing $A_1$ first, with
the desired principal symbol, as in Proposition~\ref{prop:outgoing-symb-est}. Then
\begin{equation*}\begin{aligned}
TA+AT&=TA_1^2+A_1^2 T=2A_1 TA_1+[T,A_1]A_1+A_1[A_1,T]\\
&=2A_1TA_1+[[T,A_1],A_1],
\end{aligned}\end{equation*}
and now the second term on the right hand side is two orders lower
than the first due to the
double commutator. On the other hand, by Lemma~\ref{lemma:im-tilde-P-positive},
$$
2A_1TA_1=2\sum_j A_1^*T_j^*T_j A_1+2A_1^*T_j^*T'_j A_1+2A_1^*(T'_j)^*T_j A_1+2A_1^*T''A_1,
$$
with the first term non-negative. The second and third terms are lower order by one
order of sc-decay. This is not sufficient to regard them as error
terms since they are of the same
order as the main commutator term; the same is true for the other
(namely, other than \eqref{eq:imag-control-6a}) term
$W(\sigma) A+AW(\sigma)$ of \eqref{eq:imag-control-6}. However, this is not surprising:
recall the factor $x^{-i\alpha_+(\sigma)}$ above in the asymptotics; for non-real
$\alpha_+$ the real part of the exponent is potentially large, so the
constraints on $l$ need to change just as in the real $\sigma$ case.
Now,
\begin{equation}\begin{aligned}\label{eq:im-sigma-T-expand}
-(\im\sigma)&\langle
(T(\sigma)A+AT(\sigma))u,u\rangle\\
&=-2(\im\sigma)\sum_j\|T_jA_1
u\|^2-2(\im\sigma)\sum_j \langle A_1^*T_j^*T'_j A_1
u,u\rangle\\
&\qquad\qquad-2(\im\sigma)\sum_j \langle A_1^*(T'_j)^*T_j A_1
u,u\rangle-2(\im\sigma) \langle A_1^*T'' A_1 u,u\rangle
\end{aligned}\end{equation}
The estimate
\begin{equation}\label{eq:Tjp-absorb}
|\langle A_1^*T_j^*T'_j A_1 u,u\rangle|\leq \ep\|T_j
A_1u\|^2+\ep^{-1}\|T_j'A_1 u\|^2
\end{equation}
allows, for small $\ep>0$, to absorb the first term of its right hand side into $\|T_j
A_1u\|^2$, while the second one now corresponds to $\langle A_1
(T'_j)^*T'_j A_1u,u\rangle$, and $( T'_j)^*T'_j$ has the same order as
$T''$, so it can be treated the same way, namely it is simply part of
the error term.

We now turn to the term
$W(\sigma) A+AW(\sigma)$ of \eqref{eq:imag-control-6} as well as to
the other term (other than \eqref{eq:imag-control-6}), $i[\re \tilde P(\sigma),A]$, of \eqref{eq:modified-commutator-imag}.
The operator $W(\sigma)$
has principal symbol $-2x\im(\alpha_+(\sigma))$ at the zero section by Lemma~\ref{lemma:im-tilde-P-positive}, which can be handled just as in the real $\sigma$ case. Indeed, the second term of
\eqref{eq:modified-commutator-imag} plus the $W(\sigma)$ contribution
to the first term, i.e.
$$
W(\sigma) A+AW(\sigma)+i[\re \tilde P(\sigma),A],
$$
is in $\Psiscb^{-\infty,2(\tilde r+l)-1,2l}(X)$ and has principal symbol, modulo terms
controlled by elliptic estimates (arising from $d\chi$),
\begin{equation}\begin{aligned}\label{eq:commutator-sc-version-imag-full}
x^{-2(l+\tilde r)+1}(\tau^2+\mu^2)^{\tilde
  r-3/2}\Big(4\big((l+\tilde r-\im\alpha_+(\sigma))
\tau^2&+(l+1/2-\im\alpha_+(\sigma))\mu^2\big)\\
&-4\frac{\re\sigma}{|\sigma|^2} (l+\tilde r)\tau(\tau^2+\mu^2)\Big)\chi.
\end{aligned}\end{equation}
Now, with $\chi$ chosen as discussed prior to the statement of
Lemma~\ref{lemma:im-tilde-P-positive}, so in particular with $\chi_3$
having sufficiently small support, the first and third terms of
\eqref{eq:commutator-sc-version-imag-full} can be absorbed into the
second, and thus we can write
\eqref{eq:commutator-sc-version-imag-full} as $b^2$ and take
$B\in\Psiscb^{*,l+\tilde r-1/2,l}$ with principal symbol $b$
so that
\begin{equation}\label{eq:symbolic-est-scb-imag-8a}
W(\sigma) A+AW(\sigma)+i[\re \tilde P(\sigma),A]=\mp B^*B+E+F,
\end{equation}
where $\mp$ corresponds to $\mp (l+1/2-\im\alpha_+(\sigma))>0$, $B\in\Psiscb^{*,l+\tilde r-1/2,l}$ has principal symbol $b$,
$E\in\Psiscb^{*,2(l+\tilde r)-1,l}$ arising from the $d\chi$ terms, and
$F\in\Psiscb^{*,2(l+\tilde r)-2,l}$ is lower order in the sc-decay
sense. Applying to $u$ and pairing with $u$ gives
\begin{equation}\label{eq:symbolic-est-scb-imag-16}
\|Bu\|^2\leq 2|\langle \tilde P(\sigma)u,Au\rangle|+|\langle Eu,u\rangle|+|\langle Fu,u\rangle|,
\end{equation}
and the $E$ term is controlled by elliptic estimates.

Combining \eqref{eq:im-sigma-T-expand} and
\eqref{eq:symbolic-est-scb-imag-8a}, we deduce that
\begin{equation}\begin{aligned}\label{eq:im-sigma-expand}
&\langle i(\tilde P(\sigma)^*A-A\tilde P(\sigma))u,u\rangle\\
&=-2(\im\sigma)\sum_j\|T_jA_1
u\|^2\mp\|Bu\|^2-2(\im\sigma)\sum_j \langle A_1^*T_j^*T'_j A_1
u,u\rangle\\
&\qquad\qquad-2(\im\sigma)\sum_j \langle A_1^*(T'_j)^*T_j A_1
u,u\rangle+\langle Eu,u\rangle\\
&\qquad\qquad+\langle Fu,u\rangle-2(\im\sigma) \langle A_1^*T'' A_1 u,u\rangle,
\end{aligned}\end{equation}
and the first two terms on the right hand side have matching signs under the hypotheses of
the proposition, while the $T''$ term can be absorbed into the $F$
term by modifying $F$ while keeping its order.

Now, $b$ is an elliptic multiple of $x^{1/2}a_1$,
so $\|x^{1/2}A_1 u\|^2$ is controlled by $\|Bu\|^2$ modulo terms that
can be absorbed into $|\langle Fu,u\rangle|$ (by modifying $F$ without
changing its order).
Thus, modulo terms absorbed into the $F$ term,
\begin{equation}\label{eq:imag-forcing-1a}
\langle
\tilde P(\sigma)u,Au\rangle=\langle x^{-1/2}A_1\tilde P(\sigma) u,x^{1/2}A_1 u\rangle
\end{equation}
is controlled by
\begin{equation}\label{eq:imag-forcing-1b}
\|Bu\|\|x^{-1/2}A_1\tilde P(\sigma) u\|\leq\ep
\|Bu\|^2+\ep^{-1}\|x^{-1/2}A_1\tilde P(\sigma) u\|^2,
\end{equation}
and now the first term on the right hand side can be absorbed into the
second term of the right hand side of
\eqref{eq:im-sigma-expand}. This gives, using the controlled $E$
terms, and simply dropping the term $-2(\im\sigma)\sum_j\|T_jA_1
u\|^2$ (after absorbing the third and fourth terms on the right hand
side of \eqref{eq:im-sigma-expand} into the first, using \eqref{eq:Tjp-absorb}) which matches the sign of $\mp\|Bu\|^2$, with elliptic estimates for the scattering differentiability
order, and with $r=\tilde r+l-1/2$,
\begin{equation}\label{eq:symbolic-est-scb-imagp}
\|u\|_{\Hscb^{s,r,l}}\leq C(\|\tilde P(\sigma)u\|_{\Hscb^{s-2,r+1,l+1}}+\|u\|_{\Hscb^{-N,r-1/2,l}}).
\end{equation}
Since $\|u\|_{\Hscb^{-N,r-1/2,l}}$ can be bounded by a small multiple of $\|u\|_{\Hscb^{-N,
    r,l}}$ plus a large multiple of $\|u\|_{\Hscb^{-N,-N,l}}$, with the
former being absorbable into the left hand side, this proves the
estimates of Proposition~\ref{prop:imag-control} with $l-\delta$
replaced by $l$ in the last term on the right hand side.
Again, a regularization argument shows that the
estimates hold in the stronger sense that if the right hand side is
finite, so is the left hand side.

Finally, we can use the
normal operator estimate of Lemma~\ref{lemma:simple-normal} (with the
factored out $\sigma$ being irrelevant) as in the
proof of Proposition~\ref{prop:real-sigma-Fredholm} to prove the proposition.
\end{proof}

Again, as soon
as the nullspace is trivial, the usual argument allows the last
relatively compact term to be dropped, so that
$$
\|u\|_{\Hb^{\tilde r,l}}\leq C\|\hat P(\sigma)u\|_{\Hb^{\tilde r,l+1}},
$$
as
well as
$$
\|u\|_{\Hscb^{s,r,l}}\leq C\|\hat P(\sigma)u\|_{\Hscb^{s-2,r+1,l+1}},
$$
and this is uniform for $\sigma$ in compact sets in $\{\im\sigma>0\}$
when $l$ is sufficiently negative. Since taking adjoints changes the sign of $\im\sigma$,
thus if $l>-1/2+\im(\alpha_+(\sigma))$, but $\tilde r$ is arbitrary, we still have analogous
estimates for $\hat P(\sigma)^*$, and thus Fredholm and invertibility
(in the latter case under nullspace assumptions)
statements for $\hat P(\sigma)$.

Allowing $\im\sigma\geq 0$ finally, i.e.\ considering the uniform
behavior to the reals (rather than keeping $\im\sigma$ away from $0$),
only very minor changes are needed to the proof of Proposition~\ref{prop:imag-control}, as we show below.

\begin{proof}[Proof of the general case of Theorem~\ref{thm:main}]
The cutoff
near the zero section
now becomes important, just as for real $\sigma$. We again proceed to compute
with
$$
\chi=\chi_0(\mu^2)\chi_1(\tau^2)\chi_3\Big(\frac{\tau^2}{\tau^2+\mu^2}\Big),
$$
with $\chi_0$, $\chi_1$, $\chi_3$ identically $1$ near $0$ of compact support
sufficiently close to $0$ and with $\chi_1$
having relatively large support so that $\supp\chi_0(.)\cap\supp d\chi_1(.)$
is disjoint from the zero set of $\re \hat p(\sigma)$ as above, so elliptic
scattering estimates control the $d\chi_1$ term, and $\chi_3$ also
chosen so that on the one hand elliptic sc-b estimates control the
$\supp d\chi_3(.)\cap \supp\chi_0(.)\cap\supp d\chi_1(.)$ region and on
the other hand in Lemma~\ref{lemma:commutator-sc-version-imag} the
$(l+1/2)\mu^2$ term dominates the others as discussed after that lemma. On the other hand,
doing the computation in the b-notation,
\begin{equation*}\begin{aligned}
&\Big\{\frac{\re\sigma}{|\sigma|^2}x^2(\taub^2+\mub^2)-2x\taub,\chi_0(x^2\mub^2)\Big\}\\
&=2\Big(2\frac{\re\sigma}{|\sigma|^2}x^2\taub-2x\Big)x^2\mub^2\chi_0'(x^2\mub^2)\\
&=-4x(1-\frac{\re\sigma}{|\sigma|^2}\tau)\mu^2\chi_0'(\mu^2),
\end{aligned}\end{equation*}
so if $\chi_1$ is arranged to have sufficiently small support, say in
$[-(\re\sigma)^2/2, (\re\sigma)^2/2]$, then this is non-negative. Arranging that $-\chi_0'$ is a square, this
simply adds another term of the correct, positive, sign to our symbolic computation if $\im\sigma\leq 0$ and $l+1/2>0$; it adds a term of the 
wrong sign if $\im\sigma\geq 0$ and $l+1/2<0$, but it is controlled by 
propagation estimates from the incoming radial points if $\tilde 
r+l>0$ much as in the real case. The threshold values are shifted as
in the proof of Proposition~\ref{prop:imag-control} due to the
$W(\sigma)$ terms.

The full computation proceeds exactly as above when $\im\sigma$ was
bounded away from $0$; now the terms $-2(\im\sigma)\sum_j\|T_jA_1
u\|^2$ in \eqref{eq:im-sigma-expand} are of no use (unlike before,
when they could have been used to give a stronger result, see Remark~\ref{rem:imag-control}).
The net result is again an estimate
\begin{equation}\label{eq:main-symbolic-est}
\|u\|_{\Hb^{\tilde r-1/2,l}}\leq C(\|\hat P(\sigma)u\|_{\Hb^{\tilde r-1/2,l+1}}+\|u\|_{\Hb^{-N,l}}).
\end{equation}
Now, the last term can be estimated using the normal operator as
above, yielding
$$
\|u\|_{\Hb^{\tilde r-1/2,l}}\leq C(\|\hat P(\sigma)u\|_{\Hb^{\tilde r-1/2,l+1}}+\|u\|_{\Hb^{-N,l-\delta}}).
$$
Again, as soon
as the nullspace is trivial, the usual argument allows the last
relatively compact term to be dropped, so that
$$
\|u\|_{\Hb^{\tilde r-1/2,l}}\leq C\|\hat P(\sigma)u\|_{\Hb^{\tilde r-1/2,l+1}},
$$
and this is uniform for $\sigma$ in compact sets in
$\RR\setminus\{0\}$, times $[0,R]$ along the imaginary direction with
$l$ as above.
The second microlocal version is,
under the same assumptions, with $r=\tilde r+l-1/2$,
$$
\|u\|_{\Hscb^{s,r,l}}\leq C\|\hat P(\sigma)u\|_{\Hb^{s,r+1,l+1}}.
$$
This proves Theorem~\ref{thm:main}.
\end{proof}

\section{High energy/semiclassical results}\label{sec:high}
In this final section we consider high energy scattering, which in
turn can be rescaled to a semiclassical problem. Since the arguments
are very similar to the bounded non-zero $\sigma$ ones, we only sketch
the proofs.

For the high energy estimates we need to be
more specific on the $\sigma$-dependence of $P(\sigma)$. Recall from
\eqref{eq:hat-P-form} that the conjugated operator takes the form
\begin{equation*}
\hat P(\sigma)=\hat P(0)+\sigma\hat Q-2\sigma \Big(x^2D_x+i\frac{n-1}{2}x+\frac{1}{2}x\big(-a_{00}\sigma+a_0+b_0\sigma-\sigma^{-1}a'-b'\big)\Big)
\end{equation*}
with
\begin{equation*}\begin{aligned}
\hat P(0)&=P(0)-xa'\in x^2\Diffb^2(X)+S^{-2-\delta}\Diffb^2(X),\\
\hat Q&=Q-b'x-2xa_{00}(x^2D_x)-2\sum_j xa_{0j}(xD_{y_j})\in x^2\Diffb^1(X)+S^{-2-\delta}\Diffb^1(X),
\end{aligned}\end{equation*}
where we allowed smooth dependence of $b_0,b_j,b'$ on $\sigma$. {\em From
now we assume that $b_0,b_j,b'$ are symbolic in $\sigma$, with
$b_0,b_j$ order $0$ and $b'$ order $1$, while their imaginary part is
order $-1$, resp.\ $0$.} This is the natural order:
when $\hat P(\sigma)$ is the temporal Fourier transform of a wave
operator, we expect these orders, with the imaginary part statement
coming from the formal self-adjointness of wave operators.

The
semiclassical rescaling is arrived at by multiplying $\hat P(\sigma)$
by $h^2$, where $h\in(0,1]$ is understood as comparable to $|\sigma|^{-1}$,
so $h|\sigma|$ is in a compact subset of $\Cx\setminus\{0\}$. The
rescaling gives, with $z=h\sigma$,
\begin{equation}\begin{aligned}\label{eq:hat-P-form-semi}
\hat P_h(z)&=h^2\hat P(h^{-1}z)\\
&=h^2\hat P(0)+z h\hat Q\\
&\qquad-2z
\Big(hx^2D_x+ih\frac{n-1}{2}x+\frac{1}{2}x\big(-a_{00}z+a_0 h+b_0 z-h^2 z^{-1}a'-hb'\big)\Big),
\end{aligned}\end{equation}
and now $h^2\hat P(0)$, $h\hat Q$ are semiclassical differential operators
\begin{equation*}\begin{aligned}
\hat P_h(0)&=h^2\hat P(0)=h^2 P(0)-h^2 xa'\in x^2\Diffbh^2(X)+S^{-2-\delta}\Diffbh^2(X),\\
\hat Q_h&=h\hat Q=hQ-hb'x-2xa_{00}(hx^2D_x)-2\sum_j xa_{0j}(hxD_{y_j})\\
&\qquad\qquad\qquad\qquad\in x^2\Diffbh^1(X)+S^{-2-\delta}\Diffbh^1(X),
\end{aligned}\end{equation*}
and the parenthetical final term in \eqref{eq:hat-P-form-semi} is in
$x\Diffbh^1(X)+S^{-1-\delta}\Diffbh^0(X)$. Note that terms with an
extra $h$ (beyond the $h$ incorporated into the derivatives) are semiclassically
subprincipal, while those with an extra $h^2$ are sub-subprincipal, so
for instance, modulo semiclassically sub-subprincipal terms, $\hat
P_h(0)=h^2 P(0)$, i.e.\ in the high energy limit, the potential term
$a'$ becomes, in this sense, irrelevant (with the analogous conclusion
also holding for the parenthetical final term in
\eqref{eq:hat-P-form-semi}). We remark also that the symbolic
order of $b_0,b_j$ means that the corresponding terms in $\hat Q_h$,
as well as in the final term of \eqref{eq:hat-P-form-semi},
are semiclassically principal, as is $b'$ since it comes with an extra $h^{-1}$ factor,
cancelling out the overall $h$, thus not completely
negligible. However, all these terms are subprincipal in terms of the
scattering decay, i.e.\ they vanish to leading order at $x=0$.

Now, $\hat P_h(z)$, just like the similarly defined $P_h(z)=h^2
P(h^{-1}z)$, is not semiclassically elliptic even in $x>0$. Indeed,
the semiclassical principal symbol of $P_h(z)$ is
$$
p_\semi(z)=G+z(b_0\tauh+\sum_j
b_j(\muh)_j)+zxhb'-z^2,
$$
while that of $\hat P_h(z)$ is then, corresponding to the conjugation
being a symplectomorphism at the phase space level (namely:
translation in the fibers by the differential $-\sigma x^{-2}\,dx$ of the phase, $\sigma/x$), can be arrived at
by replacing $\tauh$ by $\tauh-z$; these are elliptic at
fiber infinity, but vanish for appropriate finite $\tauh,\muh$ for $z$
real, and indeed at the zero section even if $z$ is complex. In particular,
$$
\hat p_\semi(z)|_{x=0}=\tauh^2+\muh^2-2z\tauh.
$$
The semiclassical flow structure in the scattering setting was
discussed in \cite{Vasy-Zworski:Semiclassical}; the symplectomorphism
corresponding to the conjugation simply translates this by $-\sigma
x^{-2}\,dx$. Thus, for $\hat P_h(z)$, there are two radial sets at
$x=0$, one of which is at the zero section; one of these is a source,
the other is a sink, {\em including in the sense of dynamics from the
  interior, $x>0$}.

Following \cite{Vasy-Zworski:Semiclassical}, one has semiclassical
symbolic estimates away from the scattering zero section at $\pa
X$. Thus, when the scattering decay order $r$ is above the threshold
value, i.e.\ one treats the radial set other than the zero section as the incoming one, at the
radial set other than the zero section, one gets automatic
semiclassical estimates there and one can propagate them towards the
outgoing radial set, stopping (since we need to discuss 2-microlocal
estimates) before arriving at the resolved zero section at $\pa
X$. Since we no longer have ellipticity over $X^\circ$, it is
important that for all points away from the b-front face, at $h=0$,
the flow in the backward (if the radial set outside the zero section
is a source) or forward (if the radial set outside the zero section
is a sink) direction tends to the radial set outside the zero
section. This follows from the non-trapping hypothesis in the
scattering setting. At this point it remains to do an estimate at the
zero section, acting as the outgoing radial set. For this we can use
essentially the same commutant as in the non-semiclassical setting for
the potentially non-zero $\im\sigma$ (recall that the real $\sigma$
argument used a more delicate cancellation that is not in general robust),
namely the weights, a cut off in a neighborhood of the blown up zero
section, which now must include a cutoff in the interior as
well. Thus, we take
\begin{equation*}\begin{aligned}
\chi&=\chi_0(\muh^2)\chi_1(\tauh^2)\chi_2(x)\chi_3(\tauh^2/(\tauh^2+\muh^2))\\
&=\chi_0(x^2\mubh^2)\chi_1(x^2\taubh^2)\chi_2(x)\chi_3(\taubh^2/(\taubh^2+\mubh^2)).
\end{aligned}\end{equation*}
The cutoff contributions before came from
$\chi_0(\mu^2)$ in view of the support of the differential of $\chi_1$; now $\chi_0(\muh^2)$ plays an analogous role. Since
we computed Poisson brackets using the b-structure, the contribution
came from the $x$ dependence of $\chi_0(x^2\mubh^2)$, and this $x$
dependence is completely analogous to that of $\chi_2$, with both
cutoffs being identically $1$ near $0$. Correspondingly, they
contribute with the same sign, meaning they both need to be
controlled (i.e.\ they have a sign opposite to that given by the
weight, where now $l$ is to be below the threshold value), as they are from the estimate propagated from the
incoming radial point; hence one obtains the zero section
outgoing radial estimates.

A completely analogous argument works when at
the radial set outside the zero section  one has scattering decay
order $r$ below the threshold value, and correspondingly at the zero
section, one has b-decay above the threshold regularity. In this case, one needs to start
at the zero section (which is the incoming radial set), taking
advantage of the cutoffs mentioned above having the correct sign
(matching that of the weights), and then propagate the estimates
using the non-trapping assumptions and the standard semiclassical
scattering propagation results as in
\cite{Vasy-Zworski:Semiclassical}.

In combination, these prove both semiclassical
high and low b-decay statements, namely the following analogue of
Proposition~\ref{prop:imag-control} (without using the normal operator
argument, thus no gain of decay in the error term, but with a gain in
the semiclassical parameter $h$), with the strengthened (in that
real $\sigma$ is allowed) statement
that arises in the proof of Theorem~\ref{thm:main} given after the
proposition, see \eqref{eq:main-symbolic-est}:

\begin{prop}\label{prop:semi-imag-control}
For $l<-1/2+h\im(\alpha_+(h^{-1}z))$ and $\im z\geq 0$, as well as for
$l>-1/2+h\im(\alpha_+(h^{-1}z))$ and $\im z\leq 0$, with $\tilde r,r$
arbitrary in either case, we have the estimates
$$
\|u\|_{\Hbh^{\tilde r-1/2,l}}\leq Ch^{-1}(\|\hat P_h(z)u\|_{\Hbh^{\tilde r-1/2,l+1}}+h^N\|u\|_{\Hbh^{-N,l}})
$$
and
$$
\|u\|_{\Hscbh^{s,r,l}}\leq Ch^{-1}(\|\hat P_h(z)u\|_{\Hscbh^{s-2,r+1,l+1}}+h^N\|u\|_{\Hbh^{-N,l}}).
$$
\end{prop}

Due to the $h^N$ factors, the last term on the right hand side of both
estimates of Proposition~\ref{prop:semi-imag-control} can be absorbed
into the left hand side. This again gives direct and adjoint
estimates, and one concludes that, for $h\in(0,1)$,
$$
\|u\|_{\Hbh^{\tilde r-1/2,l}}\leq Ch^{-1}\|\hat P_h(z)u\|_{\Hbh^{\tilde r-1/2,l+1}},
$$
which translates to
$$
\|u\|_{\Hbh^{\tilde r-1/2,l}}\leq Ch\|\hat P(h^{-1}z)u\|_{\Hbh^{\tilde r-1/2,l+1}},
$$
thus
$$
\|u\|_{\Hbh^{\tilde r-1/2,l}}\leq C|\sigma|^{-1}\|\hat P(\sigma)u\|_{\Hbh^{\tilde r-1/2,l+1}},
$$
with $\tilde r,l$ as before, provided that the non-trapping assumption holds.

The second microlocal version is completely analogous, as is the
complex spectral parameter version, proving Theorem~\ref{thm:high}.

\def\cprime{$'$} \def\cprime{$'$}

\end{document}